\documentclass[12pt]{article}
\usepackage{graphics}
\usepackage{graphicx}
\usepackage{color}
\usepackage{amsfonts}
\usepackage{eufrak}
\usepackage{enumerate}
\usepackage{amsmath}
\usepackage{amstext}
\usepackage{amsopn}
\usepackage{amsbsy}
\usepackage{amscd}
\usepackage{amsxtra}
\usepackage{amsthm}

\numberwithin{equation}{section}

\renewcommand{\abstractname}

\title{Uniqueness and nondegeneracy of sign-changing radial solutions to an  almost critical elliptic problem}

\newtheorem*{rk}{Remark}

\newtheorem{thm}{Theorem}[section]
\newtheorem{pop}[thm]{Proposition}
\newtheorem{lem}[thm]{Lemma}

\begin{document}

\author{Weiwei Ao\thanks{  Department of Mathematics, University of British
Columbia, Vancouver, B.C., Canada, V6T 1Z2.  Email: wwao@math.ubc.ca}
\and Juncheng Wei\thanks{ Department of Mathematics University of British
Columbia,  Vancouver, B.C., Canada, V6T 1Z2. Email: jcwei@math.ubc.ca} \and Wei Yao \thanks{ Departamento de Ingenier\'{\i}a  Matem\'atica and Centro de
Modelamiento Matem\'atico (UMI 2807 CNRS), Universidad de Chile, Casilla 170 Correo 3,
Santiago, Chile. Email: wyao.cn@gmail.com} }

\maketitle

\begin{abstract}
We study sign-changing radial solutions for the following semi-linear elliptic equation
\begin{align*}
\Delta u-u+|u|^{p-1}u=0\quad{\rm{in}}\ \mathbb{R}^N,\quad u\in H^1(\mathbb{R}^N),
\end{align*}
where $1<p<\frac{N+2}{N-2}$, $N\geq3$. It is well-known that this equation has a unique positive radial solution and  sign-changing radial solutions with exactly $k$ nodes.  In this paper, we show that such sign-changing radial solution is also unique when $p$ is close to $\frac{N+2}{N-2}$. Moreover, those solutions are non-degenerate, i.e., the kernel of the linearized operator is exactly $N$-dimensional.
\end{abstract}

\section{Introduction and main results}

In this paper we establish the uniqueness and nondegeneracy of sign-changing radially symmetric solutions to the following semi-linear elliptic equation
\begin{align}\label{eq0.1}
\Delta u-u+|u|^{p-1}u=0\quad{\rm{in}}\  \mathbb{R}^N,\quad u\in H^1(\mathbb{R}^N)
\end{align}
where $1<p<2^*-1$ and $2^*=2N/(N-2)$ is the critical Sobolev exponent for the embedding of $H^1(\mathbb{R}^N)$ into $L^{2^*}(\mathbb{R}^N),N\geq3$.
More precisely, for any $k\in\mathbb{N}$, we will prove that the following ODE problem
\begin{align}\label{eq0.1r}
\left\{\begin{array}{ll}
u''+\frac{N-1}{r}u'-u+|u|^{p-1}u=0,\quad r\in(0,\infty), \quad N\geq3,\\
u'(0)=0,\quad\lim\limits_{r\rightarrow\infty}u(r)=0,
\end{array}
\right.
\end{align}
has a unique solution $u\in C^2[0,\infty)$ such that $u(0)>0$ and $u$ has exactly $k$ zeros. Moreover, this unique solution is non-degenerate in the space of $H^1$ functions.

\medskip

Equation (\ref{eq0.1}) arises in various models in physics, mathematical physics and biology. In particular, the study of standing waves for the nonlinear Klein-Gordon or Schrodinger equations reduce to (\ref{eq0.1}).  The uniqueness and nondegeneracy of standing waves play essential role in the study of soliton dynamics or blow up for nonlinear Schr\"{o}dinger equations. We refer to the papers of Berestycki and Lions \cite{bl1}, \cite{bl2} for mathematical foundations of (\ref{eq0.1}), Rapahael \cite{rap1},  Merle and Raphael \cite{mer-rap}, Nakanishi and Schlag \cite{NS}  for backgrounds on dynamics and blow-ups of NLS.

The classical work of Gidas, Ni and Nirenberg \cite{GNN} states that all  positive solutions of (\ref{eq0.1}) are radially symmetric around some point. The uniqueness of positive solutions to (\ref{eq0.1}) has been extensively studied during the last thirty years. It was initiated by Coffman \cite{C1} with $p=3$ and $N=3$, and then improved by McLeod and Serrin \cite{MS} to $1<p\leq \frac{N}{N-2}$,  and finally extended by Kwong \cite{K} to all values of exponent $1<p<\frac{N+2}{N-2}$ by shooting method.  After these results there have been many extensions and refinements, see for example the works \cite{PS}, \cite{CL}, \cite{ST} and references therein. An essential tool in studying (\ref{eq0.1}) is the shooting method, i.e., one studies  the behavior of solutions $u(r,\alpha)$ to the initial value problem
\begin{align}\label{eq0.1i}
\left\{\begin{array}{ll}
u''+\frac{N-1}{r}u'-u+u^{p}=0,\quad r\in(0,\infty),\quad N\geq3,\\
u(0)=\alpha,\quad u'(0)=0,
\end{array}
\right.
\end{align}
for $\alpha\in(0,\infty)$ and obtains series of comparison results between two solutions to (\ref{eq0.1i}) with different initial values. One feature of their approach is that it can be extended to the $m$-Laplacian operator and more general nonlinearities, see \cite{ST} for example. However, it seems very hard to apply the approach to sign-changing solutions if one does not understand the complicated intersection between
two solutions to (\ref{eq0.1}) in the second nodal domain.

\medskip

For sign-changing radial solutions, the existence results have been established by Coffman \cite{C2} and McLeod, Troy and
Weissler \cite{MTW} using ODE shooting techniques and a
scaling argument. There are also other approaches including variational methods (Bartsch and William \cite{BW}, Struwe \cite{struwe}) and heat flow  (Conti-Verzini-Terracinni \cite{conti}). But for the uniqueness of sign-changing solutions, to our knowledge, there are few work on sign-changing solutions. In \cite{CGY}, using Coffman's approach, Cortazar, Garcia-Huidobro and Yarur study the uniqueness of sign-changing radial solution to
\begin{align}\label{eq0.1g}
\Delta u+f(u)=0,\quad{\rm{in}}\  \mathbb{R}^N,
\end{align}
under some convexity and sublinear growth conditions of $f(u)$. In the canonical case of $f(u)=|u|^{p-1}u-|u|^{q-1}u$, the condition on $p$ and $q$ is that:
\begin{align}\label{cd}
p\geq1,\ 0<q<p,\quad{\rm{and}}\ p+q\leq\frac{2}{N-2}.
\end{align}
 The result we present here is a contribution to this matter which covers the $q=1$ case and superlinear case. We shall employ a different method--the Liapunov-Schmidt reduction--to prove our result. 

\medskip

Up to now, the Lyapunov-Schmidt reduction method, which reduces an infinite-dimensional problem to a finite-dimensional one, has been widely used successfully in
constructing various solutions, see for example \cite{W}, \cite{O}, \cite{GW}, \cite{PF2}. For the uniqueness problem, Wei \cite{W} applied this method and  established the uniqueness and non-degeneracy of boundary spike solutions for the following singularly perturbed Neumann boundary problem:
\begin{align}
\left\{\begin{array}{ll}
\epsilon^2\Delta u-u+u^p=0\quad{\rm{in}}\ \Omega,\\
u>0\quad{\rm{in}}\ \Omega\ {\rm{and}}\ \frac{\partial u}{\partial\nu}=0\quad{\rm{on}}\ \partial\Omega,
\end{array}
\right.
\end{align}
where $\epsilon>0$ is a small parameter, $\Omega$ is a smooth
bounded domain in $\mathbb{R}^N$, and $p$ is subcritical. The main idea is to reduce the problem in $H^2(\Omega)$ into a finite-dimensional problem on the space of
spikes and then compute the number of critical points for a finite-dimensional problem. The same idea has been used successfully by Grossi \cite{G} in computing the number of single-peak solutions of the nonlinear Schr$\ddot{\mathrm{o}}$dinger equation
\begin{align}
\left\{\begin{array}{ll}
-\epsilon^2\Delta u+V(x)u=u^p\quad{\rm{in}}\ \mathbb{R}^N;\\
u>0,
\end{array}
\right.
\end{align}
for a suitable class of potentials $V$ and critical point $P$. But for the uniqueness problem, we do not know whether the Lyapunov-Schmidt reduction method can be used to problems other than the singularly perturbed one. 

\medskip

The purpose of this paper is to deal with the uniqueness of sign-changing radial solutions to (\ref{eq0.1}) by the Liapunov-Schmidt reduction. After setting $p=\frac{N+2}{N-2}-\epsilon$, then problem (\ref{eq0.1}) will become a singularly perturbed one and then we can use the idea in \cite{W} and \cite{DW} to establish the uniqueness and non-degeneracy of sign-changing solution to (\ref{eq0.1}) for sufficient small $\epsilon>0$.

\medskip

Our first result concerns the uniqueness of sign-changing radial solution:

\begin{thm}\label{thm-unique}
For any positive integer $k$, there exists a positive constant $\epsilon_0$ such that for
$p\in\big(\frac{N+2}{N-2}-\epsilon_0,\frac{N+2}{N-2}\big)$, there exists an unique sign-changing radial solution to (\ref{eq0.1}) with $u(0)>0$ and exactly $k$ zeros.
\end{thm}

\begin{rk}
Using the same idea, we can give a new proof on the uniqueness of positive solution to the equation (\ref{eq0.1}) with an almost critical power.
\end{rk}

\medskip

Our second result concerns the eigenvalue estimates associated with the linearized operator at $u_\epsilon$, the solutions obtained in Theorem \ref{thm-unique}:
\begin{align}
\overline{L}_\epsilon\equiv \Delta-1+p|u_\epsilon|^{p-1}.
\end{align}

We have the following non-degeneracy result:
\begin{thm}\label{thm1.2}
There exists a number $\epsilon_0>0$ such that for $p\in(\frac{N+2}{N-2}-\epsilon_0,\frac{N+2}{N-2})$, $u_\epsilon$ is non-degenerate, i.e., if $ \phi$ satisfies
\begin{equation}
\Delta \phi -\phi + p |u|^{p-1} \phi=0,  \ \ \ \ \ \  |\phi| \leq 1, \ \mbox{in} \ {\mathbb R}^N
\end{equation}
then
\begin{eqnarray}
\phi \in \mathrm{span}\left\{\frac{\partial u_\epsilon}{\partial x_1},\cdots,\frac{\partial u_\epsilon}{\partial x_N}\right\}.
\end{eqnarray}

\end{thm}

\medskip

The organization of the paper is as follows. In Section 2 we give some preliminary analysis. In Section 3 a finite dimensional reduction procedure is given. In Section 4 we show the existence and uniqueness. Finally in Section 5 we give the small eigenvalue estimate and complete the proof of Theorem \ref{thm1.2}.

\medskip

Throughout this paper we denote various generic constants by $C$. We
use $O(B),o(B)$ to mean $|O(B)|\leq C|B|,o(B)/|B|\rightarrow0$ as
$|B|\rightarrow0$, respectively.

\medskip

\noindent {\bf Acknowledgments:} The research of J. Wei is partially supported by NSERC of Canada.


%
%
%
%
%
%
%
%
%
\section{The asymptotic behavior of the solutions}

In this section, we will give some preliminary analysis. First by Pohozaev's non-existence result, equation (\ref{eq0.1}) only has trivial solution $u=0$ when $p\geq 2^*-1$, see \cite{P}. So if $u_\epsilon$ is a sign-changing solution to (\ref{eq0.1}) with $p=(N+2)/(N-2)-\epsilon,\epsilon>0$, then $u_\epsilon$ must blow up as $\epsilon\rightarrow0$. Moreover, by the result of Felmer, Quaas, Tang and Yu \cite{FQTY},
\begin{align}
u_\epsilon(0)\rightarrow\infty\quad \mathrm{and}\quad R_\epsilon\rightarrow0
\end{align}
as
$\epsilon\rightarrow0$, where $R_\epsilon$ is the first zero point
of $u_\epsilon$. Without loss of generality, in this paper,  we assume that $u_\epsilon(0)>0$ and we will consider sign-changing once radial solutions to (\ref{eq0.1}).  The proofs can be easily modified to deal with sign-changing solutions with more than one nodes.

The key estimate we shall obtain first is  the relation between $ u_\epsilon (0)$ and the first radius $R_\epsilon$. To this end, we take the so-called Emden-Fowler transformation to $u_\epsilon$ as in \cite{DW}. Let
\begin{align*}
v_\epsilon(t)=r^\alpha u_\epsilon(r),\quad r=e^t,\ \alpha=\frac{2}{p-1}.
\end{align*}
Then $v_\epsilon$ satisfies
\begin{align}\label{eq-v}
v''-\beta v'-(\gamma+e^{2t})v+|v|^{p-1}v=0,\quad
t\in(-\infty,\infty),
\end{align}
where
\begin{align}
p=\frac{N+2}{N-2}-\epsilon,\
\beta=\frac{(N-2)^2\epsilon}{4-(N-2)\epsilon}\quad\ {\rm{and}}\quad
\gamma=\frac{(N-2)^2}{4}-\frac{\beta^2}{4}.
\end{align}

Recall that the corresponding energy functional of equation
(\ref{eq0.1}) is
\begin{align*}
\widetilde{E}_\epsilon(u)=\frac{1}{2}\int_0^\infty\big(|u'|^2+|u|^2\big)r^{N-1}\,dr-\frac{1}{p+1}\int_0^\infty|u|^{p+1}r^{N-1}\,dr,
\end{align*}
and by the Emden-Fowler transformation,
\begin{align*}
\int_0^\infty |u'|^2r^{N-1}\,dr
=\int_{-\infty}^\infty\Big[|v'|^2+\gamma|v|^2\Big]e^{-\beta t}\,dt;
\end{align*}
\begin{align*}
\int_0^\infty |u|^2r^{N-1}\,dr =\int_{-\infty}^\infty
e^{2t}|v|^2e^{-\beta t}\,dt;
\end{align*}
\begin{align*}
\int_0^\infty |u|^{p+1}r^{N-1}\,dr =\int_{-\infty}^\infty
|v|^{p+1}e^{-\beta t}\,dt.
\end{align*}
Thus the corresponding energy functional of equation (\ref{eq-v})
is
\begin{align}
E_\epsilon(v)=\frac{1}{2}\int_{-\infty}^\infty\Big[|v'|^2+(\gamma+e^{2t})|v|^2\Big]e^{-\beta
t}\,dt-\frac{1}{p+1}\int_{-\infty}^\infty |v|^{p+1}e^{-\beta t}\,dt,
\end{align}
and $u(r)\in H^1(\mathbb{R}^N)$ if and only if $v(t)\in H$, where
$H$ is the Hilbert space defined by
\begin{align}
H\equiv\{v\in
H^1(\mathbb{R})|\int_{-\infty}^\infty\Big[|v'|^2+(\gamma+e^{2t})|v|^2\Big]e^{-\beta
t}\,dt<\infty\}
\end{align}
with the inner product
\begin{align}
(v,w)_\epsilon=\int_{-\infty}^\infty\Big[v'w'+(\gamma+e^{2t})vw\Big]e^{-\beta
t}\,dt.
\end{align}
Similarly, we define the weighted $L^2$-product as follows:
\begin{align}
\langle v,w\rangle_\epsilon=\int_{-\infty}^\infty vwe^{-\beta
t}\,dt.
\end{align}

To get the asymptotic behaviour of the solutions, by standard blow-up analysis, we first have the following Lemma:

\begin{lem}\label{lem-1}
Let $v_\epsilon$ be a solution of (\ref{eq-v}). Then there exists a positive constant $C=C(N)$ such that
\begin{eqnarray}
\|v_\epsilon\|_\infty\leq C.
\end{eqnarray}
\end{lem}

Since the uniqueness of positive solutions is known for $u$ in ball and annulus, so is it for $v$ and we have the following a priori estimate of energy of $v_\epsilon$:

\begin{lem}\label{lem-2}
Let $v_\epsilon$ be a solution of (\ref{eq-v}). Then there exists a small positive constant $\delta$ such that
\begin{eqnarray}\label{energe-v}
E_\epsilon(v_\epsilon)<2E_\epsilon(w_0)+\delta<3E_\epsilon(w_0),
\end{eqnarray}
where $w_0$ is the unique positive solution of the following problem
\begin{align}\label{eq-w0}
\left\{\begin{array}{ll}
w''-\frac{(N-2)^2}{4}w+w^{2^*-1}=0, w>0\quad{\rm{in}}\ \mathbb{R};\\
w(0)=\max\limits_{t\in\mathbb{R}}w(t),\quad w(t)\rightarrow0,\ {\rm{as}}\
|t|\rightarrow\infty.
\end{array}
\right.
\end{align}
\end{lem}

Using the above a priori estimate of energy, we can follow the argument of \cite{NT} to prove the following asymptotic behavior of $v_\epsilon$:

\begin{lem}\label{lem-3}
Suppose $v_\epsilon$ is a sign-changing once solution of (\ref{eq-v}), then $v_\epsilon$ has exactly one local maximum point $t_1$ and one local minimum point
$t_2$ in $(-\infty,\infty)$, provided that $\epsilon$ is
sufficiently small. Moreover,
\begin{align}\label{v}
v_\epsilon(t)=w_0(t-t_1)-w_0(t-t_2)+o(1)
\end{align}
and
\begin{align}\label{t}
t_1<t_2,\quad t_1\rightarrow-\infty,\quad t_2\rightarrow-\infty, \quad |t_2-t_1|\rightarrow\infty, \ \mbox{ as }\epsilon\to 0
\end{align}
where $w_0$ is the unique positive solution to equation
(\ref{eq-w0}) and $o(1)\rightarrow0$ as $\epsilon\rightarrow0$.
\end{lem}

\begin{proof}

First we show that the local maximum point must goes to $-\infty$ as $\epsilon\to 0$. Suppose not, there exists a sequence of local maximum points $t_\epsilon$ of $v_\epsilon$ such that $t_\epsilon\rightarrow t_0$. By the estimate of energy of $v_\epsilon$ we get $v_\epsilon(t+t_\epsilon)\rightarrow v_0$ in $C^2_{loc}$, where $v_0$ satisfies
\begin{eqnarray}\label{eq-v0-1}
v''-(\gamma_0+ e^{2t})v+v^{2^\star-1}=0,\ v\geq0\ \mathrm{in}\ \mathbb{R},\\
v(t)\rightarrow0,\ \mathrm{as}\ |t|\rightarrow\infty,\nonumber
\end{eqnarray}
where $\gamma_0=\frac{(N-2)^2}{4}$. But by Pohozave's identity, $v_0\equiv0$. This contradicts with $v_0(0)\geq\gamma_0^{1/(2^\star-1)}>0$.

\medskip

Next we show that the distance of local maximum point and zero point of $v_\epsilon$ goes to $\infty$. Suppose not, using the same notation above, there exists $d\in\mathbb{R}$ such that, $v_\epsilon(t+t_\epsilon)\rightarrow v_0$ in $C^2_{loc}((-\infty,d))$, where $v_0$ satisfies
\begin{eqnarray}\label{eq-v0-2}
v''-(\gamma_0+e^{2t})v+v^{2^\star-1}=0,\ v\geq0\ \mathrm{in}\ (-\infty,d),\\
v(d)=0,\quad v(t)\rightarrow0,\ \mathrm{as}\ |t|\rightarrow\infty. \nonumber
\end{eqnarray}
This is also a contradiction to the Pohozave's identity.

\medskip

Now we show that there only exists one local maximum point. Suppose not, there are at least are two local maximum points $t_1$ and $t_2$. We first show that $|t_1-t_2|\rightarrow\infty$. Suppose not, $|t_1-t_2|$ is bounded, then using the same notations, $v_\epsilon(t+t_1)\rightarrow v_0$ in $C^2_{loc}(\mathbb{R})$, where $v_0$ satisfies (\ref{eq-v0-1}). Moreover since $v_\epsilon'(0)=0$, $v_0'(0)=0$ and then applying Lemma 4.2 in \cite{NT} and the arguement right after the proof of Lemma 4.2, we get a contradiction. Thus $|t_1-t_2|\rightarrow\infty$. Then we have a lower bound of the energy functional $E_\epsilon(v_\epsilon)>2E_\epsilon(w_0)+C_1>2E_\epsilon(w_0)+\delta$ for some $C_1>0$ independent of $\epsilon$ small, which contradicts with Lemma \ref{lem-2}.

\medskip

For the negative part, we can get the similar result and complete the proof.

\end{proof}

Now we set
\begin{align}\label{eq0.6}
S_\epsilon[v]=v''-\beta v'-(\gamma+e^{2t})v+|v|^{p-1}v.
\end{align}

\medskip

To get more accurate information on asymptotic behaviour, we introduce the function $w$ be the unique positive solution of
\begin{align}\label{eq-w}
\left\{\begin{array}{ll}
w''-\frac{(N-2)^2}{4}w+w^p=0\quad{\rm{in}}\ \mathbb{R};\\
w(0)=\max_{t\in\mathbb{R}}w(t),\quad w(t)\rightarrow0,\ {\rm{as}}\
|t|\rightarrow\infty.
\end{array}
\right.
\end{align}
It is standard that
\begin{align}\label{w}
\left\{\begin{array}{ll}
w(t)=A_{\epsilon,N}e^{-(N-2)t/2}+O\big(e^{-p(N-2)t/2}\big),\quad t\geq0;\\\\
w'(t)=-\frac{N-2}{2}A_{\epsilon,N}e^{-(N-2)t/2}+O\big(e^{-p(N-2)t/2}\big),\quad
t\geq0,
\end{array}
\right.
\end{align}
where $A_{\epsilon,N}>0$ is a constant depending only on $\epsilon$
and $N$. Actually the function $w(t)$ can be written explicitly and
has the following form
\begin{align}
w(t)=\gamma_0^{\frac{1}{p-1}}(\frac{p+1}{2})^\frac{1}{p-1}
\Big[\cosh\big(\frac{p-1}{2}\gamma_0^{1/2}t\big)\Big]^{-\frac{2}{p-1}}.\nonumber
\end{align}

\medskip

Testing (\ref{w}) with $w$ and $w'$ and integrating by parts, one arrives at the following identity:
\begin{eqnarray}\label{ide-w}
\int_{\mathbb{R}}|w'|^2\,dt=(\frac{1}{2}-\frac{1}{p+1})\int_{\mathbb{R}}w^{p+1}\,dt
=\gamma_0(\frac{p-1}{p+3})\int_{\mathbb{R}}w^2\,dt.
\end{eqnarray}

\medskip

Note that $w\notin H$ when $N=3,4$. For $t_1, t_2$ obtained in Lemma \ref{lem-3},  we set $w_{j,t_j}$ to be the unique solution of
\begin{align}
v''-(\gamma_0+e^{2s})v+w^p_{t_j}=0,\quad\mathrm{where}\ w_{t_j}(s)=w(s-t_j), \ j=1,2
\end{align}
in the Hilbert space $H$. The existence and uniqueness of $w_{j,t_j}$ are derived from the Riesz's representation theorem.

\medskip

Using the ODE analysis, we can obtain the following asymptotic expansion of
$w_{j,t_j}$, $j=1,2$ for whose proof we postpone to Appendix A:

\begin{lem}\label{lem-4}
For $\epsilon$ sufficient small,
\begin{align}
w_{j,t_j}=w_{t_j}+\phi_{j,t_j}+O(e^{2t_j}),
\end{align}
where for $N=3$,
\begin{align}\label{n=3}
\phi_{j,t_j}(s)=-e^{t_j/2}A_{\epsilon,3}\,e^{-s/2}\big(1-e^{-e^s}\big);
\end{align}
for $N=4$,
\begin{align}\label{n=4}
\phi_{j,t_j}(s)=-e^{t_j}A_{\epsilon,4}\,e^{-s}\big[1-\rho_0(\frac{1}{4}e^{2s})\big],
\end{align}
and
\begin{align*}
\rho_0(r)=2\sqrt{r}K_1(2\sqrt{r}),
\end{align*}
where $K_1(z)$ is the modified Bessel function of second kind and
satisfies
\begin{align*}
z^2K_1''(z)+zK_1'(z)-(z^2+1)K_1(z)=0;
\end{align*}
for  $N=5$,
\begin{align}\label{n=5}
\phi_{j,t_j}(s)=-e^{3t_j/2}A_{\epsilon,5}\,e^{-3s/2}\Big[1-(1+e^s)
 e^{-e^s}\Big];
\end{align}
for $N=6$,
\begin{align*}
\phi_{j,t_j}=-e^{2t_j}A_{\epsilon,6}e^{-2s}\big[1-u_0(\frac{1}{16^2}e^{4s})\big],
\end{align*}
where
\begin{align*}
u_0(r)=8\sqrt{r}K_2(4r^{1/4}),
\end{align*}
where $K_2(z)$ is the modified Bessel function of second kind and
satisfies
\begin{align*}
z^2K_2''(z)+zK_2'(z)-(z^2+4)K_2(z)=0;
\end{align*}
for  $N>6$, $\phi_{j,t_j}=0$.
\end{lem}

\begin{rk}
By the maximum principle, we have the following useful estimates:
\begin{eqnarray}\label{w-phi-1}
0<w_{j,t_j}<w_{t_j},\quad
-w_{t_j}<\phi_{j,t_j}<0,
\end{eqnarray}
and
\begin{eqnarray}\label{w-phi-2}
|w_{t_j}'|\leq c_1w_{t_j}\leq c_2,\ |w_{j,t_j}'|\leq c_1w_{j,t_j}\leq c_2,
\end{eqnarray}
where $c_1,c_2$ are two positive constants independent of $\epsilon $ small.
\end{rk}

From the above Lemma \ref{lem-4} and (\ref{t}), we see that $w_{j,t_j}=w_{t_j}+o(1)=w_{0,t_j}+o(1)$ in all
the cases for $j=1,2$. Thus by (\ref{v}),
\begin{align}
v_\epsilon(t)=w_{\epsilon,\textbf{t}}+o(1),
\end{align}
where
\begin{align}\label{eqw}
w_{\epsilon,\textbf{t}}(t)=w_{1,t_1}(t)-w_{2,t_2}(t).
\end{align}

Before studying the properties of $w_{\epsilon,\textbf{t}}$, we need some preliminary Lemmas. The first one is a useful inequality:

\begin{lem}
For $x\geq0$, $y\geq0$,
\begin{eqnarray}\label{ieq-1}
|x^p-y^p|\leq\left\{
\begin{array}{ll}
|x-y|^p,\quad&\mathrm{if}\ 0<p<1,\\
p|x-y|(x^{p-1}+y^{p-1}),\quad&\mathrm{if}\ 1\leq p<\infty.
\end{array}
\right.
\end{eqnarray}
\end{lem}

The following Lemma is proved in Proposition 1.2 of \cite{BL}.

\begin{lem}\label{lem-B}
Let $f\in C(R)\cap L^\infty(R)$, $g\in C(R)$ be even and satisfy for some $\alpha\geq 0$, $\beta\geq 0$, $\gamma_0\in R$,
\begin{eqnarray*}
&&f(x)exp(\alpha|x|)|x|^\beta\to \gamma_0 \mbox{ as } |x|\to \infty,\\
&&\int_{R}|g(x)|exp(\alpha|x|)(1+|x|^\beta)dx<\infty.
\end{eqnarray*}
Then
\begin{equation*}
exp(\alpha|y|)|y|^\beta\int_{R}g(x+y)f(x)dx\to \gamma_0\int_{R}g(x)exp(-\alpha x_1)dx
\mbox{ as }|y|\to \infty .
\end{equation*}
\end{lem}

Next we state a useful Lemma about the interactions of two $w$'s:

\begin{lem}\label{lem2.1}
For $|r-s|\gg1$ and $\eta>\theta>0$, there hold
\begin{eqnarray}\label{eq2.11}
w^\eta(t-r)w^\theta(t-s)=O(w^\theta(|r-s|));
\end{eqnarray}
\begin{eqnarray}\label{eq2.12}
\int_{-\infty}^\infty
w^\eta(t-r)w^\theta(t-s)\,dt=\big(1+o(1)\big)w^\theta(|r-s|)
\int_{-\infty}^\infty w^\eta(t)e^{\theta\sqrt{\gamma_0}t}\,dt,
\end{eqnarray}
where $o(1)\rightarrow0$ as $|t-s|\rightarrow\infty$.
\end{lem}

\begin{proof}
The conclusion follows from (\ref{w}) and the Lebesgue's Dominated
Convergence Theorem, see for example \cite{MNW}.

\end{proof}

Now we have the following error estimates:
\begin{lem}\label{lem2.3}
For $\epsilon$ sufficiently small and $t_1,t_2$ satisfy (\ref{t}),
there is a constant $C$ independent of $\epsilon,t_1$ and $t_2$ such that
\begin{eqnarray*}
\|S_\epsilon[w_{\epsilon,\textbf{t}}]\|_\infty+\int_{-\infty}^\infty\Big|
S_\epsilon[w_{\epsilon,\textbf{t}}]\Big|e^{-\beta t}\,dt\leq
C\big[\beta+e^{\tau t_2}+e^{-\tau|t_1-t_2|/2}\big]\quad{\rm{for}}\
N=3;
\end{eqnarray*}
\begin{eqnarray*}
\|S_\epsilon[w_{\epsilon,\textbf{t}}]\|_\infty+\int_{-\infty}^\infty\Big|
S_\epsilon[w_{\epsilon,\textbf{t}}]\Big|e^{-\beta t}\,dt\leq
C\big[\beta+t_2^{\tau}e^{2\tau t_2}+e^{-\tau|t_1-t_2|}\big]\quad{\rm{for}}\
N=4;
\end{eqnarray*}
\begin{eqnarray*}
\|S_\epsilon[w_{\epsilon,\textbf{t}}]\|_\infty+\int_{-\infty}^\infty\Big|
S_\epsilon[w_{\epsilon,\textbf{t}}]\Big|e^{-\beta t}\,dt\leq
C\big[\beta+e^{2\tau t_2}+e^{-\tau(N-2)|t_1-t_2|/2}\big]\quad{\rm{for}}\
N\geq5,
\end{eqnarray*}
where $\tau$ is a constant satisfying $\frac{1}{2}<\tau<\frac{\min\{p,2\}}{2}$.
\end{lem}

\begin{proof}
By the equation satisfied by $w_{j,t_j}$, we have
\begin{align}\label{s[w]}
S_\epsilon[w_{\epsilon,\textbf{t}}] =-\beta
w_{\epsilon,\textbf{t}}'-(\gamma-\gamma_0)w_{\epsilon,\textbf{t}}
+|w_{\epsilon,\textbf{t}}|^{p-1}w_{\epsilon,\textbf{t}}
-w_{t_1}^p+w_{t_2}^p.
\end{align}
From the exponential decay of $w_j$, (\ref{w-phi-1}), (\ref{w-phi-2}) and $\gamma-\gamma_0=-\frac{\beta^2}{4}$, we deduce that
\begin{align*}
|\beta w_{\epsilon,\textbf{t}}'|\leq C\beta(w_{t_1}+w_{t_2}),
\end{align*}
and
\begin{align*}
|(\gamma-\gamma_0)w_{\epsilon,\textbf{t}}|\leq C\beta^2(w_{t_1}+w_{t_2}).
\end{align*}

Next, we divide $(-\infty,\infty)$ into 2 intervals $I_1,I_2$
defined by
$$I_1=(-\infty,\frac{t_1+t_2}{2}),\quad I_2=[\frac{t_1+t_2}{2},\infty).$$
Then on $I_i$, $i=1,2$, we have $w_{t_j}\leq w_{t_i}$ and then $w_{j,t_j}\leq w_{i,t_i}$ by the maximum principle for $i\neq j$. So on $I_1$ we use inequality (\ref{ieq-1}) to get
\begin{eqnarray*}
\Big|\big|w_{\epsilon,\textbf{t}}\big|^{p-1}w_{\epsilon,\textbf{t}}-w_{t_1}^p+w_{t_2}^p\Big|
&\leq Cw_{t_1}^{p-1}w_{t_2}+Cw_{t_1}^{p-1}\phi_{1,t_1}\\
&\leq Cw_{t_1}^{p-\tau}w_{t_2}^\tau+Cw_{t_1}^{p-\tau}\phi_{1,t_1}^\tau,
\end{eqnarray*}
for any $\tau\in(0,1]$.

\medskip

Similarly on $I_2$ the following inequality holds,
\begin{eqnarray*}
\Big|\big|w_{\epsilon,\textbf{t}}\big|^{p-1}w_{\epsilon,\textbf{t}}-w_{t_1}^p+w_{t_2}^p\Big|
&\leq Cw_{t_2}^{p-1}w_{t_1}+Cw_{t_2}^{p-1}\phi_{2,t_2}\\
&\leq Cw_{t_2}^{p-\tau}w_{t_1}^\tau+Cw_{t_2}^{p-\tau}\phi_{2,t_2}^\tau,
\end{eqnarray*}
for any $\tau\in(0,1]$.

\medskip

By the above inequalities and using Lemma \ref{lem-B}, the desired result follows.

\end{proof}

In order to obtain the a priori estimate of $t_1,t_2$ and compute the energy expansion $E_\epsilon[w_{\epsilon,\textbf{t}}]$, we need to estimate
\begin{align*}
\|v_\epsilon-w_{\epsilon,\textbf{t}}\|_\infty\quad \mathrm{and}\ \|v_\epsilon-w_{\epsilon,\textbf{t}}\|_H.
\end{align*}

\begin{lem}\label{phi}
For $\epsilon$ sufficiently small, there is a constant $C$ independent of $\epsilon$ such that
\begin{align}
v_\epsilon=w_{\epsilon,\textbf{t}}+\phi_\epsilon,
\end{align}
where
\begin{align}\label{eq1}
\left\{\begin{array}{lll}
\|\phi_\epsilon\|_\infty+\|\phi_\epsilon\|_{H}\leq
C\big[\beta+e^{\tau t_2}+e^{-\tau|t_1-t_2|/2}\big]\quad{\rm{for}}\
N=3;\\\\
\|\phi_\epsilon\|_\infty+\|\phi_\epsilon\|_{H}\leq
C\big[\beta+t_2^{\tau}e^{2\tau t_2}+e^{-\tau|t_1-t_2|}\big]\quad{\rm{for}}\
N=4;\\\\
\|\phi_\epsilon\|_\infty+\|\phi_\epsilon\|_{H}\leq
C\big[\beta+e^{2\tau t_2}+e^{-\tau(N-2)|t_1-t_2|/2}\big]\quad{\rm{for}}\
N\geq5,
\end{array}
\right.
\end{align}
where $\tau$ satisfies $\frac{1}{2}<\tau<\frac{\min\{p,2\}}{2}$.
\end{lem}

\begin{proof}

We may follow the arguments given in the proof of Lemma 2.4 in \cite{LWY}. First by the properties of $w_{j,t_j}$'s we can choose proper $t_j$'s such that the maximum points $r_\epsilon$ and minimum points $s_\epsilon$ of $v_\epsilon$ are also the ones of $w_{\epsilon,\textbf{t}}$, respectively. Let $v_\epsilon=w_{\epsilon,\textbf{t}}+\phi_\epsilon$,
then $\phi_\epsilon\rightarrow0$ and satisfies
\begin{align*}
\phi''-\beta\phi'-(\gamma+e^{2t})\phi+p|w_{\epsilon,\textbf{t}}|^{p-1}\phi
 +S_\epsilon[w_{\epsilon,\textbf{t}}]+N_\epsilon[\phi]=0,
\end{align*}
where
\begin{align*}
N_\epsilon[\phi]=|w_{\epsilon,\textbf{t}}+\phi|^{p-1}(w_{\epsilon,\textbf{t}}+\phi)
 -|w_{\epsilon,\textbf{t}}|^{p-1}w_{\epsilon,\textbf{t}}-p|w_{\epsilon,\textbf{t}}|^{p-1}\phi.
\end{align*}
Now we prove the estimates for $\phi_\epsilon$ by contradiction. Denote the right hand side order term of (\ref{eq1}) by $K_\epsilon$ and suppose that
\begin{align*}
\|\phi_\epsilon\|_\infty/K_\epsilon\rightarrow\infty.
\end{align*}
Let $\widetilde{\phi}_\epsilon=\phi_\epsilon/\|\phi_\epsilon\|_\infty$, then
$\widetilde{\phi}_\epsilon$ satisfies
\begin{align}\label{phi-1-0}
\widetilde{\phi}''-\beta\widetilde{\phi}'-(\gamma+e^{2t})\widetilde{\phi}
 +p\big|w_{\epsilon,\textbf{t}}\big|^{p-1}\widetilde{\phi}
 +\frac{S_\epsilon[w_{\epsilon,\textbf{t}}]}{\|\phi_\epsilon\|_\infty}
 +\frac{N_\epsilon[\phi_\epsilon]}{\|\phi_\epsilon\|_\infty}=0.
\end{align}
Note that
\begin{align}\label{phi-1-1}
\Big|\frac{S_\epsilon[w_{\epsilon,\textbf{t}}]}{\|\phi_\epsilon\|_\infty}\Big|\leq CK_\epsilon/\|\phi_\epsilon\|_\infty,\quad
\Big|\frac{N_\epsilon[\phi_\epsilon]}{\|\phi_\epsilon\|_\infty}\Big|\leq C\|\phi_\epsilon\|_\infty^{\min\{p-1,1\}}.
\end{align}

\medskip

Let $t_\epsilon$ be such that $\widetilde{\phi}_\epsilon(t_\epsilon)=\|\widetilde{\phi}_\epsilon\|_\infty=1$ (the same proof applies if $\widetilde{\phi}_\epsilon(t_\epsilon)=-1$). Then by (\ref{phi-1-0}), (\ref{phi-1-1}) and the Maximum Principle, we have $|t_\epsilon-t_1|\leq C$ or $|t_\epsilon-t_2|\leq C$. Thus $|t_\epsilon-r_\epsilon|\leq C$ or $|t_\epsilon-s_\epsilon|\leq C$. Without loss of generality, we assume that $|t_\epsilon-r_\epsilon|\leq C$. Then by the usual elliptic regular theory, we may take a subsequence
$\widetilde{\phi}_\epsilon(t+r_\epsilon)\rightarrow\widetilde{\phi}_0(t)$ as
$\epsilon\rightarrow0$ in $C^1_{{\rm{loc}}}(\mathbb{R})$ since $|r_\epsilon-t_1|\rightarrow0$, where
$\widetilde{\phi}_0$ satisfies
\begin{align*}
\widetilde{\phi}_0''-\frac{(N-2)^2}{4}\widetilde{\phi}_0
 +\frac{N+2}{N-2}w_0^{4/(N-2)}\widetilde{\phi}_0=0,\ {\rm{and}}\
\widetilde{\phi}_0'(0)=0,
\end{align*}
which implies $\widetilde{\phi}_0\equiv0$. This contradicts to the fact that $1=\widetilde{\phi}_\epsilon(t_\epsilon)\rightarrow\widetilde{\phi}_0(t_0)$ for some $t_0$. Therefore we complete the proof.

\end{proof}

The following is the basic technical estimate in this paper which gives the a priori estimates for $t_1$ and $t_2$:

\begin{lem}\label{basic lemma}
For $\epsilon$ sufficient small we have for $N=3$,
\begin{align}
\left\{\begin{array}{ll}
t_1=\log a+2\log b+3\log\beta;\\
t_2=\log a+\log\beta,
\end{array}
\right.
\end{align}
where $a,b$ are constants and
\begin{align*}
a\rightarrow a_{0,3},\ b\rightarrow b_{0,3}.
\end{align*}
Here $a_{0,3},b_{0,3}$ are positive constants.

For $N=4$,
\begin{align}
\left\{\begin{array}{ll}
t_1-t_2=\log b+\log\beta;\\
-2t_2e^{2t_2}=a\beta,
\end{array}
\right.
\end{align}
where $a,b$ are constants and
\begin{align*}
a\rightarrow a_{0,4},\ b\rightarrow b_{0,4}.
\end{align*}
Here $a_{0,4},b_{0,4}$ are positive constants.

For $N\geq5$,
\begin{align}
\left\{\begin{array}{ll}
t_1=\frac{1}{2}\log a+\frac{2}{N-2}\log b+\frac{N+2}{2(N-2)}\log\beta;\\
t_2=\frac{1}{2}\log a+\frac{1}{2}\log\beta,
\end{array}
\right.
\end{align}
where $a,b$ are constants and
\begin{align*}
a\rightarrow a_{0,N},\ b\rightarrow b_{0,N}.
\end{align*}
Here $a_{0,N},b_{0,N}$ are positive constants.
\end{lem}

\begin{proof}
Here we only give the proof for $N=3$, for $N>3$, the proof is similar.  From $S_\epsilon[v_\epsilon]=0$ and $v_\epsilon=w_{\epsilon,\textbf{t}}+\phi_\epsilon$ we deduce
\begin{align}\label{eq3.1a}
L_{\epsilon,\textbf{t}}[\phi]+S_\epsilon[w_{\epsilon,\textbf{t}}]+N_\epsilon[\phi]=0,
\end{align}
where
\begin{align*}
L_{\epsilon,\textbf{t}}[\phi]=\phi''-\beta\phi'-(\gamma+e^{2t})\phi+p|w_{\epsilon,\textbf{t}}|^{p-1}\phi,
\end{align*}
and
\begin{align*}
N_\epsilon[\phi]=|w_{\epsilon,\textbf{t}}+\phi|^{p-1}(w_{\epsilon,\textbf{t}}+\phi)
 -|w_{\epsilon,\textbf{t}}|^{p-1}w_{\epsilon,\textbf{t}}-p|w_{\epsilon,\textbf{t}}|^{p-1}\phi.
\end{align*}

Multiplying (\ref{eq3.1a}) by $w_{1,t_1}'$ and integrating over $\mathbb{R}$, we
obtain
\begin{align*}
\int_{-\infty}^\infty
L_{\epsilon,\textbf{t}}[\phi]w_{1,t_1}'\,dt+\int_{-\infty}^\infty
S_\epsilon[w_{\epsilon,\textbf{t}}]w_{1,t_1}' +\int_{-\infty}^\infty
N_\epsilon[\phi]w_{1,t_1}'\,dt=0.
\end{align*}
Integrating by parts and using Lemma \ref{phi} we have
\begin{align*}
\int_{-\infty}^\infty L_{\epsilon,\textbf{t}}[\phi]w_{1,t_1}'\,dt
=&\int_{-\infty}^\infty\Big[w_{1,t_1}'''+\beta w_{1,t_1}''-(\gamma+e^{2t})w_{1,t_1}'
 +p|w_{\epsilon,\textbf{t}}|^{p-1}w_{1,t_1}'\Big]\phi\,dt\nonumber\\
=&\int_{-\infty}^\infty\beta\big[(\gamma_0+e^{2t})w_{1,t_1}-w_{t_1}^p\big]\phi\,dt
 -(\gamma-\gamma_0)\int_{-\infty}^\infty w_{1,t_1}'\phi\,dt\\
&\quad+2\int_{-\infty}^\infty e^{2t}w_{1,t_1}\phi\,dt
 +p\int_{-\infty}^\infty\Big[|w_{\epsilon,\textbf{t}}|^{p-1}w_{1,t_1}'
 -w_{t_1}^{p-1}w_{t_1}'\Big]\phi\,dt\nonumber\\
=&o(\big[\beta+e^{t_1}+e^{t_2}+e^{-|t_1-t_2|/2}\big]).\nonumber
\end{align*}
Similarly we can obtain
\begin{align*}
\int_{-\infty}^\infty L_{\epsilon,\textbf{t}}[\phi]w_{2,t_2}'\,dt
=o\big(\big[\beta+e^{t_1}+e^{t_2}+e^{-|t_1-t_2|/2}\big]\big).
\end{align*}

For the nonlinearity term using (\ref{ieq-1}) we get
\begin{align*}
\Big|N_\epsilon[\phi]\Big|&=\Big||w_{\epsilon,\textbf{t}}+\phi|^{p-1}(w_{\epsilon,\textbf{t}}+\phi)
\\
&-|w_{\epsilon,\textbf{t}}|^{p-1}w_{\epsilon,\textbf{t}}-p|w_{\epsilon,\textbf{t}}|^{p-1}\phi\Big|\leq
C|\phi|^{\min\{p,2\}}.
\end{align*}
So using the exponential decay of $w$ and taking $\tau>\max\{\frac{1}{2},\frac{1}{p}\}$ we deduce
\begin{align*}
\int_{-\infty}^\infty
N_\epsilon[\phi]w_{1,t_1}'\,dt=O(\|\phi\|_\infty^2)=o(\big[\beta+e^{t_1}+e^{t_2}+e^{-|t_1-t_2|/2}\big]).
\end{align*}
Similarly we can obtain
\begin{align*}
\int_{-\infty}^\infty
N_\epsilon[\phi]w_{2,t_2}'\,dt=o(\big[\beta+e^{t_1}+e^{t_2}+e^{-|t_1-t_2|/2}\big]).
\end{align*}

To estimate $\int_{-\infty}^\infty S_\epsilon[w_{\epsilon,\textbf{t}}]w_{1,t_1}'\,dt$, we write
\begin{align*}
\int_{-\infty}^\infty S_\epsilon[w_{\epsilon,\textbf{t}}]w_{1,t_1}'\,dt
=&\int_{-\infty}^\infty \Big[-\beta w_{\epsilon,\textbf{t}}'
 -(\gamma-\gamma_0)w_{\epsilon,\textbf{t}}
 +|w_{\epsilon,\textbf{t}}|^{p-1}w_{\epsilon,\textbf{t}}-w_{t_1}^p+w_{t_2}^p\Big]w_{1,t_1}'\,dt\nonumber\\
=&\,E_1+E_2+E_3,
\end{align*}
where
\begin{align*}
&E_1=-\beta\int_{-\infty}^\infty
 w_{\epsilon,\textbf{t}}'w_{1,t_1}'\,dt;\\
&E_2=-(\gamma-\gamma_0)\int_{-\infty}^\infty
 w_{\epsilon,\textbf{t}}w_{1,t_1}'\,dt;\\
&E_3=\int_{-\infty}^\infty
 \Big[|w_{\epsilon,\textbf{t}}|^{p-1}w_{\epsilon,\textbf{t}}
 -w_{t_1}^p+w_{t_2}^p\Big]w_{1,t_1}'\,dt.
\end{align*}
Using (\ref{w-phi-2}) and Lemma \ref{lem2.1} we obtain
\begin{eqnarray*}
E_1=-\beta\int_{-\infty}^\infty|w_{1,t_1}'|^2\,dt+\beta\int_{-\infty}^\infty
w_{1,t_1}'w_{2,t_2}'\,dt =-\beta\int_{-\infty}^\infty
|w'|^2\,dt+o(\beta).
\end{eqnarray*}
Note that $\gamma-\gamma_0=-\beta^2/4$ and using (\ref{w-phi-2}) we get
\begin{eqnarray*}
E_2=\frac{\beta^2}{4}\int_{-\infty}^\infty
w_{\epsilon,\textbf{t}}w_{1,t_1}'\,dt=O(\beta^2).
\end{eqnarray*}

To estimate $E_3$, following the argument in the proof of Lemma \ref{lem2.3}. We divide $(-\infty,\infty)$ into two intervals $I_1,I_2$
defined by
$$I_1=(-\infty,\frac{t_1+t_2}{2}),\quad I_2=[\frac{t_1+t_2}{2},\infty).$$
 On $I_1$ the following equality holds:
\begin{eqnarray*}
&&\big|w_{\epsilon,\textbf{t}}\big|^{p-1}w_{\epsilon,\textbf{t}}-w_{t_1}^p+w_{t_2}^p\\
&&=\big[(w_{1,t_1}-w_{2,t_2})^p-w_{1,t_1}^p+pw_{1,t_1}^{p-1}w_{2,t_2}\big]-pw_{1,t_1}^{p-1}w_{2,t_2}\\
&&+\big[(w_{t_1}+\phi_{1,t_1})^p-w_{t_1}^p-pw_{t_1}^{p-1}\phi_{1,t_1}\big]+pw_{t_1}^{p-1}\phi_{1,t_1}.
\end{eqnarray*}
We use inequality (\ref{ieq-1}) to get
\begin{eqnarray*}
\Big|(w_{1,t_1}-w_{2,t_2})^p-w_{1,t_1}^p+pw_{1,t_1}^{p-1}w_{2,t_2}\Big|
\leq Cw_{t_1}^{p-\delta}w_{t_2}^\delta,\\
\Big|(w_{t_1}+\phi_{1,t_1})^p-w_{t_1}^p-pw_{t_1}^{p-1}\phi_{1,t_1}\Big|
\leq Cw_{t_1}^{p-\delta}\phi_{1,t_1}^\delta,
\end{eqnarray*}
for any $1<\delta<2$.

\medskip

Then using Lemma \ref{lem2.1} and integrating by parts, we get
\begin{align*}
&\int_{I_1}\Big[\big|w_{\epsilon,\textbf{t}}\big|^{p-1}w_{\epsilon,\textbf{t}}
 -w_{t_1}^p+w_{t_2}^p\Big]w_{1,t_1}'\,dt\\
&=\int_{-\infty}^\infty w_{t_1}^pw_{t_2}'\,dt-\int_{-\infty}^\infty w_{t_1}^p\phi_{1,t_1}'\,dt
+o(e^{-|t_1-t_2|/2})+o(e^{t_1}).
\end{align*}
On the other hand, on $I_2$, using $w_{1,t_1}\leq w_{2,t_2}$, (\ref{w-phi-2}) and inequality (\ref{ieq-1}) we get
\begin{eqnarray*}
\Big|\big[\big|w_{\epsilon,\textbf{t}}\big|^{p-1}w_{\epsilon,\textbf{t}}-w_{t_1}^p+w_{t_2}^p\big]w_{1,t_1}'\Big|
\leq Cw_{t_1}^{\delta}w_{t_2}^{p+1-\delta}+C\phi_{2,t_2}^{\delta}w_{t_2}^{p+1-\delta},
\end{eqnarray*}
for any $1<\delta<2$.

\medskip

Using Lemma \ref{lem2.1} we get
\begin{align*}
\int_{I_2}\Big[\big|w_{\epsilon,\textbf{t}}\big|^{p-1}w_{\epsilon,\textbf{t}}
 -w_{t_1}^p+w_{t_2}^p\Big]w_{t_1}'\,dt
=o(e^{-|t_1-t_2|/2})+o(e^{t_2}).
\end{align*}
Thus
\begin{eqnarray*}
E_3&=\frac{1}{2}e^{-|t_1-t_2|/2}A_{\epsilon,3}\int_{-\infty}^\infty w^pe^{t/2}\,dt
+\frac{1}{2}e^{t_1}A_{\epsilon,3}\int_{-\infty}^\infty
w^pe^{t/2}\,dt\\
&+o(e^{-|t_1-t_2|/2})+o(e^{t_2}),
\end{eqnarray*}
and then
\begin{eqnarray}
\int_{-\infty}^\infty
S_\epsilon[w_{\epsilon,\textbf{t}}]w_{1,t_1}'\,dt
&=-\beta\int_{-\infty}^\infty|w'|^2\,dt+\frac{1}{2}e^{-|t_1-t_2|/2}A_{\epsilon,3}\int_{-\infty}^\infty w^pe^{t/2}\,dt\\
&+o(\beta)+o(e^{-|t_1-t_2|/2})+o(e^{t_2}).\nonumber
\end{eqnarray}
Similarly,
\begin{eqnarray}
\int_{-\infty}^\infty
S_\epsilon[w_{\epsilon,\textbf{t}}]w_{2,t_2}'\,dt
&=\beta\int_{-\infty}^\infty|w'|^2\,dt+\frac{1}{2}\Big[e^{-|t_1-t_2|/2}-e^{t_2}\Big]A_{\epsilon,3}\int_{-\infty}^\infty
w^pe^{t/2}\,dt\\
&+o(\beta)+o(e^{-|t_1-t_2|/2})+o(e^{t_2}).\nonumber
\end{eqnarray}
Combining all the estimates above, one can see that $\beta,e^{t_2}$ and $e^{-|t_1-t_2|/2}$ must have the same order.

\medskip

Therefore,
\begin{eqnarray*}
\left\{\begin{array}{ll} -\beta\int_{-\infty}^\infty |w'|^2\,dt
+\frac{1}{2}e^{-|t_1-t_2|/2}A_{\epsilon,3}\int_{-\infty}^\infty w^pe^{t/2}\,dt
=o(\beta);\\
\\
\beta\int_{-\infty}^\infty
|w'|^2\,dt+\frac{1}{2}\Big[e^{-|t_1-t_2|/2}-e^{t_2}\Big]A_{\epsilon,3}\int_{-\infty}^\infty w^pe^{t/2}\,dt =o(\beta).
\end{array}
\right.
\end{eqnarray*}
Let
\begin{eqnarray*}
e^{t_2}=a\beta,\quad e^{-|t_1-t_2|/2}=b\beta,
\end{eqnarray*}
then
\begin{eqnarray*}
a=\frac{4\int_{-\infty}^\infty |w'|^2\,dt}{A_{\epsilon,3}\int_{-\infty}^\infty
w^pe^{t/2}\,dt}+o(1)\rightarrow a_{0,3},
\end{eqnarray*}
and
\begin{eqnarray*}
b=\frac{2\int_{-\infty}^\infty |w'|^2\,dt}{A_{\epsilon,3}\int_{-\infty}^\infty
w^pe^{t/2}\,dt}+o(1)\rightarrow b_{0,3},
\end{eqnarray*}
where $a_{0,3},b_{0,3}$ are positive constants.

\medskip

Thus
\begin{align*}
\left\{\begin{array}{ll}
t_1=\log a+2\log b+3\log\beta;\\
t_2=\log a+\log\beta,
\end{array}
\right.
\end{align*}
where
\begin{eqnarray*}
a\rightarrow a_{0,3},\ b\rightarrow b_{0,3}.
\end{eqnarray*}

\end{proof}

\medskip

We now introduce the following configuration space:
\begin{align}
\Lambda\equiv
\left\{\begin{array}{lll}
\Big\{\textbf{t}=(t_1,t_2)|\frac{1}{2}a_{0,3}\beta<e^{t_2}<\frac{3}{2}a_{0,3}\beta,
\frac{1}{2}b_{0,3}\beta<e^{(t_1-t_2)/2}<\frac{3}{2}b_{0,3}\beta\Big\}\ {\rm{for}}\ N=3;\\\\
\Big\{\textbf{t}=(t_1,t_2)|\frac{1}{2}a_{0,4}\beta<-2t_2e^{2t_2}<\frac{3}{2}a_{0,4}\beta,
\frac{1}{2}b_{0,4}\beta<e^{t_1-t_2}<\frac{3}{2}b_{0,4}\beta\Big\}\ {\rm{for}}\ N=4;\\\\
\Big\{\textbf{t}=(t_1,t_2)|\frac{1}{2}a_{0,N}\beta<e^{2t_2}<\frac{3}{2}a_{0,N}\beta,
\frac{1}{2}b_{0,N}\beta<e^{(N-2)(t_1-t_2)/2}<\frac{3}{2}b_{0,N}\beta\Big\}\ {\rm{for}}\ N\geq5.
\end{array}
\right.
\end{align}

Then by Lemma \ref{basic lemma}, for $\epsilon$ sufficient small, $\textbf{t}=(t_1,t_2)\in\Lambda$ if $v_\epsilon$
is a sign-changing solution to equation (\ref{eq-v}). In the next section, we will show an one-to-one correspandence between the sign-changing solution of (\ref{eq0.1}) and the critical points of some functional in $\Lambda$.

\section{The existence result: Liapunov-Schmidt reduction}

In this section we outline the main steps of the so called Liapunov-Schmidt reduction method or localized energy method, which reduces the infinite problem to finding a critical point for a functional on a finite dimensional space. A very important observation is the reduction Lemma \ref{crucial lemma}. To achieve this, we first study the solvability of a linear problem and then apply some standard fixed point theorem for contraction mapping to solve the nonlinear problem. Since the procedure has been by now standard (see for example \cite{MNW} and the references therein), we will omit most of the details.

\subsection{An auxiliary linear problem}

In this subsection we study a linear theory which allows us to perform
the finite-dimensional reduction procedure.

First observing that
orthogonality to $\frac{\partial w_{\epsilon,\textbf{t}}}{\partial
t_j}$ in $H$, $j=1,2$, is equivalent to orthogonality to the following functions
\begin{align}\label{z-1-1}
Z_{\epsilon,t_j}:=-(\partial_{t_j}w_{\epsilon,\textbf{t}})''+\beta(\partial_{t_j}w_{\epsilon,\textbf{t}})'+(\gamma+e^{2t})\partial_{t_j}w_{\epsilon,\textbf{t}},\quad j=1,2,
\end{align}
in the weighted $L^2$-product $\langle\ ,\ \rangle_\epsilon\,$.

\medskip

By (\ref{eqw}) and elementary computations, we obtain for $j=1,2$,
\begin{align}
\partial_{t_j}w_{\epsilon,\textbf{t}}=(-1)^{j+1}\partial_{t_j}w_{j,t_j}=(-1)^{j+1}\big(\partial_{t_j}w_{t_j}+\partial_{t_j}\phi_{t,t_j}\big)+O(e^{2t_j}),
\end{align}
and
\begin{align}\label{Z}
Z_{\epsilon,t_j}&=(-1)^{j}\Big[pw_{t_j}^{p-1}w_{t_j}'
 -\beta(\partial_{t_j}w_{j,t_j})'-(\gamma-\gamma_0)\partial_{t_j}w_{j,t_j}\Big].
\end{align}

In this section, we consider the following linear problem:
\begin{align}\label{eq3.1}
\left\{\begin{array}{ll}
L_{\epsilon,\textbf{t}}[\phi]:=\phi''-\beta\phi'-(\gamma+e^{2t})\phi+p|w_{\epsilon,\textbf{t}}|^{p-1}\phi
=h+\sum\limits^2_{j=1}c_jZ_{\epsilon,t_j};\\\\
\langle\phi,Z_{\epsilon,t_j}\rangle_\epsilon=0,\ j=1,2,
\end{array}
\right.
\end{align}
where ${\bf t}\in \Lambda$.

\medskip

For the above linear problem, we have the following result:

\begin{pop}\label{pop3.1}
Let $\phi$ satisfy (\ref{eq3.1}) with $\|h\|_\infty<\infty$. Then for $\epsilon$ sufficiently
small, we have
\begin{eqnarray}\label{eq3.6}
\|\phi\|_\infty\leq C\|h\|_\infty,
\end{eqnarray}
where $C$ is a positive constant independent of $\epsilon$ and
$\textbf{t}\in\Lambda$.
\end{pop}

\begin{proof}
The proof is now standard, see for example \cite{MNW}.

\end{proof}

\medskip

Using Fredholm's alternative we can show the following existence result:
\begin{pop}\label{pop3.2}
There exists $\epsilon_0>0$ such that for any $\epsilon<\epsilon_0$
the following property holds true. Given $h\in
L^\infty(\mathbb{R})$, there exists a unique pair $(\phi,c_1,c_2)$
such that
\begin{align}\label{eq3.16}
\left\{\begin{array}{ll}
L_{\epsilon,\textbf{t}}[\phi]=h+\sum\limits^2_{j=1}c_jZ_{\epsilon,t_j};\\
\langle\phi,Z_{\epsilon,t_j}\rangle_\epsilon=0,\ j=1,2.
\end{array}
\right.
\end{align}
Moreover, we have
\begin{align}\label{eq3.18}
\|\phi\|_\infty+\sum\limits^2_{j=1}|c_j|\leq C\|h\|_\infty
\end{align}
for some positive constant $C$.
\end{pop}

\begin{proof}
The result follows from Proposition \ref{pop3.1} and the Fredholm's alternative theorem, see for example \cite{MNW}.
%

\end{proof}

\medskip

In the following , if $\phi$ is the unique solution given in
Proposition \ref{pop3.2}, we set
\begin{eqnarray}\label{eq3.20}
\phi=A_\epsilon(h).
\end{eqnarray}
Note that (\ref{eq3.18}) implies
\begin{eqnarray}\label{eq3.21}
\|A_\epsilon(h)\|_\infty\leq C\|h\|_\infty.
\end{eqnarray}

\subsection{The nonlinear projected problem}
This subsection is devoted to the solvability of the following non-linear projected problem:
\begin{align}\label{eq4.2}
\left\{\begin{array}{ll}
(w_{\epsilon,\textbf{t}}+\phi)''-\beta(w_{\epsilon,\textbf{t}}+\phi)'
-(\gamma+e^{2t})(w_{\epsilon,\textbf{t}}+\phi)
+|w_{\epsilon,\textbf{t}}+\phi|^{p-1}(w_{\epsilon,\textbf{t}}+\phi)=\sum\limits_{j=1}^2c_jZ_{\epsilon,t_j};\\\\
\langle\phi,Z_{\epsilon,t_j}\rangle_\epsilon=0,\ j=1,2.
\end{array}
\right.
\end{align}

The first equation in (\ref{eq4.2}) can be written as
\begin{align*}
\phi''-\beta\phi'-(\gamma+e^{2t})\phi+p|w_{\epsilon,\textbf{t}}|^{p-1}\phi
=-S_\epsilon[w_{\epsilon,\textbf{t}}]-N_\epsilon[\phi]+\sum^2_{j=1}c_jZ_{\epsilon,t_j},
\end{align*}
where
\begin{align}\label{eq4.4}
N_\epsilon[\phi]=|w_{\epsilon,\textbf{t}}+\phi|^{p-1}(w_{\epsilon,\textbf{t}}+\phi)
-|w_{\epsilon,\textbf{t}}|^{p-1}w_{\epsilon,\textbf{t}}
-p|w_{\epsilon,\textbf{t}}|^{p-1}\phi_{\epsilon,\textbf{t}}.
\end{align}

\medskip

First, we have the following estimates:
\begin{lem}\label{lem4.1}
For $\textbf{t}\in\Lambda$ and $\epsilon$ sufficiently small, we
have for $\|\phi\|_\infty+\|\phi_1\|_\infty+\|\phi_2\|_\infty\leq1$,
\begin{align}\label{eq4.5}
\|N_\epsilon[\phi]\|_\infty\leq C\|\phi\|_\infty^{\min\{p,2\}};
\end{align}
\begin{align}\label{eq4.6}
\|N_\epsilon[\phi_1]-N_\epsilon[\phi_2]\|_\infty\leq
C(\|\phi_1\|_\infty^{\min\{p-1,1\}}+\|\phi_2\|_\infty^{\min\{p-1,1\}})\|\phi_1-\phi_2\|_\infty.
\end{align}
\end{lem}

\begin{proof}
These inequalities follows from the mean-value theorem and
inequality (\ref{ieq-1}).
\end{proof}

By the standard fixed point theorem for contraction mapping and Implicit Function Theorem,  Lemma \ref{lem2.3} and \ref{lem4.1}, we have the following Proposition:
\begin{pop}\label{pop4.2}
For $\textbf{t}\in\Lambda$ and $\epsilon$ sufficiently small, there
exists a unique $\phi=\phi_{\epsilon,\textbf{t}}$ such that
(\ref{eq4.2}) holds. Moreover,
$\textbf{t}\rightarrow\phi_{\epsilon,\textbf{t}}$ is of class $C^1$
as a map into $H$, and we have
\begin{align}\label{eq4.7}
\|\phi_{\epsilon,\textbf{t}}\|_\infty+\sum\limits^2_{j=1}|c_j|\leq\left\{\begin{array}{lll}
C\big[\beta+e^{\tau t_2}+e^{-\tau|t_1-t_2|/2}\big]\quad{\rm{for}}\
N=3;\\\\
C\big[\beta+t_2^{\tau}e^{2\tau t_2}+e^{-\tau|t_1-t_2|}\big]\quad{\rm{for}}\
N=4;\\\\
C\big[\beta+e^{2\tau t_2}+e^{-\tau(N-2)|t_1-t_2|/2}\big]\quad{\rm{for}}\
N\geq5,
\end{array}
\right.
\end{align}
where $\tau$ satisfies $\frac{1}{2}<\tau<\frac{\min\{p,2\}}{2}$.
\end{pop}

\subsection{Energy expansion for reduced energy functional}

In this subsection we expand the quantity
\begin{align}\label{eq5.1}
K_\epsilon(\textbf{t})=E_\epsilon[w_{\epsilon,\textbf{t}}+\phi_{\epsilon,\textbf{t}}]:\
\Lambda\rightarrow \mathbb{R}
\end{align}
in terms of $\epsilon$ and $\textbf{t}$, where $\phi_{\epsilon,\textbf{t}}$
is obtained in Proposition \ref{pop4.2}.

\begin{lem}\label{lem-K}
For $\textbf{t}\in\Lambda$ and $\epsilon$ sufficiently small, we
have for $N=3$,
\begin{align*}
K_\epsilon(\textbf{t})&=
\big(\frac{1}{2}-\frac{1}{p+1}\big)(e^{-\beta t_1}+e^{-\beta
t_2})\int_{-\infty}^\infty
w^{p+1}\,dt+\frac{1}{2}e^{t_2}A_{\epsilon,3}\int_{-\infty}^\infty w^{p}e^{t/2}\,dt\nonumber\\
 &\quad+e^{-|t_1-t_2|/2}A_{\epsilon,3}\int_{-\infty}^\infty
w^pe^{t/2}\,dt+o(\beta)+o(e^{t_2})+o(e^{-|t_1-t_2|/2})\\
&=\widetilde{K}_\epsilon(\textbf{t})+o(\beta)+o(e^{t_2})+o(e^{-|t_1-t_2|/2}).
\end{align*}

For $N=4$,
\begin{align*}
K_\epsilon(\textbf{t})&=
\big(\frac{1}{2}-\frac{1}{p+1}\big)(e^{-\beta t_1}+e^{-\beta
t_2})\int_{-\infty}^\infty
w^{p+1}\,dt-\frac{1}{4}t_2e^{2t_2}A_{\epsilon,4}\int_{-\infty}^\infty w^{p}e^{t}\,dt\nonumber\\
 &\quad+e^{-|t_1-t_2|}A_{\epsilon,4}\int_{-\infty}^\infty w^pe^{t}\,dt+o(\beta)+o(t_2e^{2t_2})+o(e^{-|t_1-t_2|})\\
&=\widetilde{K}_\epsilon(\textbf{t})+o(\beta)+o(t_2e^{2t_2})+o(e^{-|t_1-t_2|}).
\end{align*}

For $N\geq5$,
\begin{align*}
K_\epsilon(\textbf{t})&=
\big(\frac{1}{2}-\frac{1}{p+1}\big)(e^{-\beta t_1}+e^{-\beta
t_2})\int_{-\infty}^\infty
w^{p+1}\,dt+\frac{1}{2}e^{2t_2}\int_{-\infty}^\infty w^{2}e^{2t}\,dt\nonumber\\
 &\quad+e^{-(N-2)|t_1-t_2|/2}A_{\epsilon,N}\int_{-\infty}^\infty
  w^pe^{(N-2)t/2}\,dt+o(\beta)+o(e^{2t_2})+o(e^{-(N-2)|t_1-t_2|/2})\\
 &=\widetilde{K}_\epsilon(\textbf{t})+o(\beta)+o(e^{2t_2})+o(e^{-(N-2)|t_1-t_2|/2}).
\end{align*}
\end{lem}

\begin{proof}
Here again, we only give the proof for $N=3$. By the definition of $K_\epsilon({\bf t})$, we can re-write it as
\begin{align}\label{K0}
K_\epsilon(\textbf{t})=E_\epsilon[w_{\epsilon,\textbf{t}}]+K_1+K_2-K_3,
\end{align}
where
\begin{align*}
K_1=\int_{-\infty}^\infty\Big[w_{\epsilon,\textbf{t}}'\phi_{\epsilon,\textbf{t}}'
+(\gamma+e^{2t})w_{\epsilon,\textbf{t}}\phi_{\epsilon,\textbf{t}}\Big]e^{-\beta
t}\,dt
-\int_{-\infty}^\infty|w_{\epsilon,\textbf{t}}|^{p-1}w_{\epsilon,\textbf{t}}\phi_{\epsilon,\textbf{t}}e^{-\beta
t}\,dt;
\end{align*}
\begin{align*}
K_2=\frac{1}{2}\int_{-\infty}^\infty\Big[|\phi_{\epsilon,\textbf{t}}'|^2
+(\gamma+e^{2t})|\phi_{\epsilon,\textbf{t}}|^2
-p|w_{\epsilon,\textbf{t}}|^{p-1}|\phi_{\epsilon,\textbf{t}}|^2\Big]e^{-\beta
t}\,dt;
\end{align*}
\begin{align*}
K_3&=\frac{1}{p+1}\int_{-\infty}^\infty\Big[|w_{\epsilon,\textbf{t}}+\phi_{\epsilon,\textbf{t}}|^{p+1}
-|w_{\epsilon,\textbf{t}}|^{p+1}
-(p+1)|w_{\epsilon,\textbf{t}}|^{p-1}w_{\epsilon,\textbf{t}}\phi_{\epsilon,\textbf{t}}
\\
&-\frac{1}{2}p(p+1)|w_{\epsilon,\textbf{t}}|^{p-1}|\phi_{\epsilon,\textbf{t}}|^2
\Big]e^{-\beta t}\,dt.
\end{align*}

Integrating by parts and using Lemma \ref{lem2.3}, \ref{phi}, we have
\begin{align}\label{K1}
|K_1|=\Big|-\int_{-\infty}^\infty
S_\epsilon[w_{\epsilon,\textbf{t}}]\phi_{\epsilon,\textbf{t}}e^{-\beta
t}\,dt\Big|
=o\Big(\big[\beta+e^{t_2}+e^{-|t_1-t_2|/2}\big]\Big).
\end{align}

\medskip

To estimate $K_2$, note that $\phi_{\epsilon,\textbf{t}}$
satisfies
\begin{align}\label{eq5.4}
&\phi_{\epsilon,\textbf{t}}''-\beta\phi_{\epsilon,\textbf{t}}'-(\gamma+e^{2t})\phi_{\epsilon,\textbf{t}}\nonumber\\
&=-|w_{\epsilon,\textbf{t}}+\phi_{\epsilon,\textbf{t}}|^{p-1}(w_{\epsilon,\textbf{t}}+\phi_{\epsilon,\textbf{t}})
+|w_{\epsilon,\textbf{t}}|^{p-1}w_{\epsilon,\textbf{t}}-S_\epsilon[w_{\epsilon,\textbf{t}}]
+\sum_{j=1}^2c_jZ_{\epsilon,t_j}.
\end{align}
Integrating by parts and using the orthogonality condition (\ref{eq4.2}), we have
\begin{align*}
2K_2&=\int_{-\infty}^\infty\Big[|w_{\epsilon,\textbf{t}}+\phi_{\epsilon,\textbf{t}}|^{p-1}(w_{\epsilon,\textbf{t}}+\phi_{\epsilon,\textbf{t}})
-|w_{\epsilon,\textbf{t}}|^{p-1}w_{\epsilon,\textbf{t}}\\
&-p|w_{\epsilon,\textbf{t}}|^{p-1}\phi_{\epsilon,\textbf{t}}
+S_\epsilon[w_{\epsilon,\textbf{t}}]\Big]
\phi_{\epsilon,\textbf{t}}e^{-\beta t}\,dt.
\end{align*}
By the mean value theorem and inequality (\ref{ieq-1}) we get
\begin{align*}
\Big||w_{\epsilon,\textbf{t}}+\phi_{\epsilon,\textbf{t}}|^{p-1}(w_{\epsilon,\textbf{t}}+\phi_{\epsilon,\textbf{t}})
-|w_{\epsilon,\textbf{t}}|^{p-1}w_{\epsilon,\textbf{t}}
-p|w_{\epsilon,\textbf{t}}|^{p-1}\phi_{\epsilon,\textbf{t}}
\Big|\leq
C|\phi_{\epsilon,\textbf{t}}|^{\min\{p,2\}}.
\end{align*}
So using Lemma \ref{lem2.3} and \ref{phi} we deduce
\begin{align}\label{K2}
K_2=o\Big(\big[\beta+e^{t_2}+e^{-|t_1-t_2|/2}\big]\Big).
\end{align}

\medskip

For $K_3$, using the mean value theorem and inequality (\ref{ieq-1}),
\begin{align*}
&\Big||w_{\epsilon,\textbf{t}}+\phi_{\epsilon,\textbf{t}}|^{p+1}
-|w_{\epsilon,\textbf{t}}|^{p+1}
-(p+1)|w_{\epsilon,\textbf{t}}|^{p-1}w_{\epsilon,\textbf{t}}\phi_{\epsilon,\textbf{t}}
\\
&-\frac{1}{2}p(p+1)|w_{\epsilon,\textbf{t}}|^{p-1}|\phi_{\epsilon,\textbf{t}}|^2\Big|
\leq C|\phi_{\epsilon,\textbf{t}}|^{\min\{p+1,3\}}.
\end{align*}
So, again, using Lemma \ref{lem2.3} and \ref{phi} it follows that
\begin{align}\label{K3}
K_3=o\Big(\big[\beta+e^{t_2}+e^{-|t_1-t_2|/2}\big]\Big).
\end{align}
Combing with (\ref{K0}), (\ref{K1}), (\ref{K2}), (\ref{K3}), and the estimates in Appendix B, we obtain the desired estimates.

\end{proof}

\medskip

We will end this section with a reduction lemma which is important for both the existence and uniqueness:

\begin{lem}\label{crucial lemma}
$v_{\epsilon,\textbf{t}}=w_{\epsilon,\textbf{t}}+\phi_{\epsilon,\textbf{t}}$ is a
critical point of $E_\epsilon$ if and only if $\textbf{t}$ is a
critical point of $K_\epsilon$ in $\Lambda$.
\end{lem}

\begin{proof}
The proof follows from the proofs in \cite{GWW}, \cite{W}. For the sake of completeness, we include a proof here.

\medskip

By Proposition \ref{pop4.2}, there exists an $\epsilon_0$ such that,
for $0<\epsilon<\epsilon_0$, we have a $C^1$ map
$\textbf{t}\rightarrow\phi_{\epsilon,\textbf{t}}$ from $\Lambda$
into $H$ such that
\begin{align}\label{eq-v-phi}
S_\epsilon[v_{\epsilon,\textbf{t}}]=\sum_{j=1}^2c_j(\textbf{t})Z_{\epsilon,t_j},\quad v_{\epsilon,\textbf{t}}=w_{\epsilon,\textbf{t}}+\phi_{\epsilon,\textbf{t}},
\end{align}
for some constants $c_j$, which also are of class $C^1$ in
$\textbf{t}$.

\medskip

First integrating by parts we get
\begin{align}\label{equiv-1}
\partial_{t_j}K_\epsilon(\textbf{t})
&=\int_{-\infty}^\infty\Big[v_{\epsilon,\textbf{t}}'\big(\partial_{t_j}w_{\epsilon,\textbf{t}}
 +\partial_{t_j}\phi_{\epsilon,\textbf{t}}\big)'+(\gamma+e^{2t})v_{\epsilon,\textbf{t}}
 \big(\partial_{t_j}w_{\epsilon,\textbf{t}}+\partial_{t_j}\phi_{\epsilon,\textbf{t}}\big)
 \Big]e^{-\beta t}\,dt\nonumber\\
&\quad\quad-\int_{-\infty}^\infty\big|v_{\epsilon,\textbf{t}}\big|^{p-1}v_{\epsilon,\textbf{t}}
 \big(\partial_{t_j}w_{\epsilon,\textbf{t}}+\partial_{t_j}\phi_{\epsilon,\textbf{t}}\big)e^{-\beta t}\,dt\\
&=-\int_{-\infty}^\infty
S_\epsilon[v_{\epsilon,\textbf{t}}]
\big(\partial_{t_j}w_{\epsilon,\textbf{t}}+\partial_{t_j}\phi_{\epsilon,\textbf{t}}\big)e^{-\beta
t}\,dt.\nonumber
\end{align}

If $v_{\epsilon,\textbf{t}}=w_{\epsilon,\textbf{t}}+\phi_{\epsilon,\textbf{t}}$
is a critical point of $E_\epsilon$, then $S_\epsilon[v_{\epsilon,\textbf{t}}]=0$. By (\ref{equiv-1}) we get
\begin{align*}
\partial_{t_j}K_\epsilon(\textbf{t})=-\int_{-\infty}^\infty
S_\epsilon[v_{\epsilon,\textbf{t}}]\big(\partial_{t_j}w_{\epsilon,\textbf{t}}+\partial_{t_j}\phi_{\epsilon,\textbf{t}}\big)e^{-\beta
t}\,dt=0,
\end{align*}
which means that $\textbf{t}$ is a critical point of $K_\epsilon$.

\medskip

On the other hand, let $\textbf{t}_\epsilon\in\Lambda$ be a critical point of $K_\epsilon$, that is $\partial_{t_j}K_\epsilon(\textbf{t}_\epsilon)=0$, $j=1,2$, by (\ref{equiv-1}) we get
\begin{align*}
0=\partial_{t_j}K_\epsilon(\textbf{t}_\epsilon)=-\int_{-\infty}^\infty
S_\epsilon[v_{\epsilon,\textbf{t}_\epsilon}]
\big(\partial_{t_j}w_{\epsilon,\textbf{t}_\epsilon}+\partial_{t_j}\phi_{\epsilon,\textbf{t}_\epsilon}\big)e^{-\beta
t}\,dt
\end{align*}
for $j=1,2$.
Hence by (\ref{eq-v-phi}) we have
\begin{align*}
\sum_{i=1}^2c_i(\textbf{t}_\epsilon)\int_{-\infty}^\infty
Z_{\epsilon,t_{\epsilon,i}}
\big(\partial_{t_j}w_{\epsilon,\textbf{t}_\epsilon}+\partial_{t_j}\phi_{\epsilon,\textbf{t}_\epsilon}\big)e^{-\beta
t}\,dt=0.
\end{align*}
By Proposition \ref{pop4.2} and  the fact
$\langle\phi_{\epsilon,\textbf{t}_\epsilon},Z_{\epsilon,t_{\epsilon,i}}\rangle_\epsilon=0$,
\begin{align}\label{phi-z}
\langle
Z_{\epsilon,t_{\epsilon,i}},\partial_{t_j}\phi_{\epsilon,\textbf{t}_\epsilon}\rangle_\epsilon
=-\langle\phi_{\epsilon,\textbf{t}_\epsilon},\partial_{t_j}Z_{\epsilon,t_{\epsilon,i}}\rangle_\epsilon
=o(1).
\end{align}
On the other hand,
\begin{align}\label{z-w}
\int_{-\infty}^\infty Z_{\epsilon,t_{\epsilon,i}}
\partial_{t_j}w_{\epsilon,\textbf{t}_\epsilon}e^{-\beta
t}\,dt=\langle
Z_{\epsilon,t_{\epsilon,i}},\partial_{t_j}w_{\epsilon,\textbf{t}_\epsilon}\rangle_\epsilon
=\delta_{ij}p\int_{-\infty}^\infty w^{p-1}|w'|^2\,dt+o(1).
\end{align}
By (\ref{phi-z}) and (\ref{z-w}), the matrix
$$\int_{-\infty}^\infty
Z_{\epsilon,t_{\epsilon,i}}
\big(\partial_{t_j}w_{\epsilon,\textbf{t}_\epsilon}+\partial_{t_j}\phi_{\epsilon,\textbf{t}_\epsilon}\big)e^{-\beta
t}\,dt$$ is diagonally dominant and thus is non-singular, which
implies $c_i(\textbf{t}_\epsilon)=0$ for $i=1,2$. Hence
$v_{\epsilon,\textbf{t}_\epsilon}=w_{\epsilon,\textbf{t}_\epsilon}+\phi_{\epsilon,\textbf{t}_\epsilon}$ is a critical point of $E_\epsilon$. This finishes the proof.

\end{proof}

\begin{rk}
Note that in the proof the theorem, we assume that the solution
$v_\epsilon$ of equation (\ref{eq-v}) can be written as
$v_\epsilon=w_{\epsilon,\textbf{t}}+\phi_\epsilon$ with
$\phi_\epsilon$ satisfying
\begin{align}\label{ortho}
\langle\phi_\epsilon,Z_{\epsilon,t_j}\rangle_\epsilon=0,\ j=1,2.
\end{align}

\medskip

In general, using (\ref{z-w}) we can decompose
\begin{align*}
\phi_\epsilon=\overline{\phi_\epsilon}+\sum_{j=1}^2
d_j\partial_{t_j}w_{\epsilon,\textbf{t}},
\end{align*}
where $\overline{\phi_\epsilon}$ satisfies (\ref{ortho}) and
$d_j=O(\|\phi_\epsilon\|_\infty)$. Thus we can write
\begin{align*}
v_\epsilon=w_{\epsilon,\textbf{t}}+\sum_{j=1}^2 d_j\partial_{t_j}w_{\epsilon,\textbf{t}}+\overline{\phi_\epsilon}
\end{align*}
and get the desired result using the same argument for $w_{\epsilon,\textbf{t}}+\sum\limits_{j=1}^2d_j\partial_{t_j}w_{\epsilon,\textbf{t}}.$
\end{rk}

\section{The uniqueness result}

By Lemma \ref{crucial lemma}, the number of sign-changing once solutions of (\ref{eq-v}) equals to the number of critical points of $K_\epsilon(\textbf{t})$. To count the number of critical points of $K_\epsilon(\textbf{t})$, we need to compute $\partial K_\epsilon(\textbf{t})$ and $\partial^2
K_\epsilon(\textbf{t})$.

\medskip

Recall that $K_\epsilon(\textbf{t})$ and $\widetilde{K}_\epsilon(\textbf{t})$ are defined in (\ref{eq5.1}) and Lemma \ref{lem-K}. The crucial estimate to prove uniqueness of $v_\epsilon$ and $u_\epsilon$ is the following Proposition:

\begin{pop}\label{pop4.1}
$K_\epsilon(\textbf{t})$ is of $C^2$ in $\Lambda$ and for $\epsilon$
sufficiently small, we have
\begin{enumerate}[(1)]
\item $K_\epsilon(\textbf{t})-\widetilde{K}_\epsilon(\textbf{t})=o(\beta)$;

\item $\partial K_\epsilon(\textbf{t})-\partial \widetilde{K}_\epsilon(\textbf{t})=o(\beta)$ uniformly for $\textbf{t}\in\Lambda$;

\item if $\textbf{t}_\epsilon\in\Lambda$ is a critical point of $K_\epsilon$, then
\begin{eqnarray}\label{diff2}
\partial^2K_\epsilon(\textbf{t}_\epsilon)-\partial^2
\widetilde{K}_\epsilon(\textbf{t}_\epsilon)=o(\beta).
\end{eqnarray}
\end{enumerate}
\end{pop}

\medskip

The proof of Proposition \ref{pop4.1} will be delayed until the end of this section. Let us now use it to prove the uniqueness of $v_\epsilon$.

\medskip
\noindent
{\bf Proof of theorem \ref{thm-unique}}
By lemma \ref{crucial lemma}, we just need to prove that $K_\epsilon(\textbf{t})$ has only one critical point in $\Lambda$. We prove it in the following steps as in \cite{W}.

\medskip

{\bf Step 1.} By (2) of proposition \ref{pop4.1}, both $K_\epsilon(\textbf{t})$ and $\widetilde{K}_\epsilon(\textbf{t})$ have no critical points on $\partial\Lambda$ and a continuous deformation argument shows that $\partial K_\epsilon(\textbf{t})$ has the same degree as $\partial\widetilde{K}_\epsilon(\textbf{t})$ on $\Lambda$. By the definition of $\widetilde{K}_\epsilon(\textbf{t})$, we have
 $\deg(\widetilde{K}_\epsilon(\textbf{t}),\Lambda,0)=1$ and thus $\deg(\partial K_\epsilon(\textbf{t}),\Lambda,0)=1$.

\medskip

{\bf Step 2.} At each critical point $\textbf{t}_\epsilon$ of $K_\epsilon(\textbf{t})$, we have
\begin{eqnarray*}
\deg(\partial K_\epsilon(\textbf{t}),\Lambda\cap B_{\delta_\epsilon}(\textbf{t}_\epsilon),0)=1,
\end{eqnarray*}
for $\delta_\epsilon$ sufficiently small.

\medskip

This follows from (3) of Proposition \ref{pop4.1} and the fact that the eigenvalues of the matrix $\beta^{-1}(\partial_{t_i}\partial_{t_j}\widetilde{K}_\epsilon(\textbf{t}_\epsilon))$ are positive and away from $0$.

\medskip

{\bf Step 3.} From step 2, we deduce that $K_\epsilon(\textbf{t})$ has only a finite number of critical points in $\Lambda$, say, $k_\epsilon$. By the properties of degree, we have
\begin{eqnarray*}
\deg(\partial K_\epsilon(\textbf{t}),\Lambda, 0)=k_\epsilon.
\end{eqnarray*}

\medskip

By step 1, $k_\epsilon=1$ and then Theorem \ref{thm-unique} is thus proved.

\qed

\noindent
In the rest of this section, we shall prove Proposition \ref{pop4.1}.

\medskip
\noindent{\bf Proof of Proposition \ref{pop4.1}}
The proof of part (1) follows from Lemma \ref{lem-K}. We now prove part (2) of Proposition \ref{pop4.1} as follows:
\begin{align}\label{eq4.1a}
\begin{array}{lll}
\partial_{t_j} K_\epsilon(\textbf{t})&=\int\limits_{-\infty}^\infty\Big[v_{\epsilon,\textbf{t}}'
(\partial_{t_j}v_{\epsilon,\textbf{t}})'+(\gamma+e^{2t})v_{\epsilon,\textbf{t}}
\partial_{t_j}v_{\epsilon,\textbf{t}}\Big]e^{-\beta t}\,dt
 -\int\limits_{-\infty}^\infty|v_{\epsilon,\textbf{t}}|^{p-1}v_{\epsilon,\textbf{t}}
 \partial_{t_j}v_{\epsilon,\textbf{t}}e^{-\beta t}\,dt\\\\
&=-\int\limits_{-\infty}^\infty S_\epsilon[v_{\epsilon,\textbf{t}}]\partial_{t_j}v_{\epsilon,\textbf{t}}e^{-\beta t}\,dt\\\\
&=J_1+J_2,
\end{array}
\end{align}
where
\begin{align*}
J_1\equiv-\int\limits_{-\infty}^\infty S_\epsilon[w_{\epsilon,\textbf{t}}+\phi_{\epsilon,\textbf{t}}]\partial_{t_j}w_{\epsilon,\textbf{t}}e^{-\beta t}\,dt,
\end{align*}
and
\begin{align*}
J_2\equiv-\int\limits_{-\infty}^\infty S_\epsilon[w_{\epsilon,\textbf{t}}+\phi_{\epsilon,\textbf{t}}]\partial_{t_j}\phi_{\epsilon,\textbf{t}}e^{-\beta t}\,dt.
\end{align*}

\medskip

Using similar argument as in Lemma \ref{basic lemma},for $N=3$, we can obtain
\begin{align}\label{I-1}
J_1=\left\{\begin{array}{ll}
-\beta\int\limits_{-\infty}^\infty|w'|^2\,dt+\frac{1}{2}e^{(t_1-t_2)/2}A_{\epsilon,3}\int\limits_{-\infty}^\infty w^pe^{t/2}\,dt+o(\beta), j=1;\\\\
-\beta\int\limits_{-\infty}^\infty|w'|^2\,dt-\frac{1}{2}\Big[e^{(t_1-t_2)/2}-e^{t_2}\Big] A_{\epsilon,3}\int\limits_{-\infty}^\infty w^pe^{t/2}\,dt+o(\beta),j=2.
\end{array}
\right.
\end{align}

By (\ref{eq4.2}) and Proposition \ref{pop4.2},
\begin{align}\label{I-2}
\begin{array}{ll}
J_2&=-\sum\limits_{i=1}^2c_i(\textbf{t})\int\limits_{-\infty}^\infty
 Z_{\epsilon,t_i}\partial_{t_j}\phi_{\epsilon,\textbf{t}}e^{-\beta t}\,dt\\\\
&=\sum\limits_{i=1}^2c_i(\textbf{t})\int\limits_{-\infty}^\infty
\phi_{\epsilon,\textbf{t}}\partial_{t_j}Z_{\epsilon,t_i}e^{-\beta t}\,dt=o(\beta).
\end{array}
\end{align}

\medskip

Combining the above two estimates (\ref{I-1}) and (\ref{I-2}), part (2) of Proposition \ref{pop4.1} is thus proved.

\medskip

In the rest we shall prove part (3) of Proposition \ref{pop4.1}. Using (\ref{eq4.1a}),
\begin{align*}
\begin{array}{ll}
\partial_{t_i}\partial_{t_j}K_\epsilon(\textbf{t})&=
 \partial_{t_i}\Big[-\int\limits_{-\infty}^\infty S_\epsilon[v_{\epsilon,\textbf{t}}]\partial_{t_j}v_{\epsilon,\textbf{t}}e^{-\beta t}\,dt\Big]\\\\
&=-\int\limits_{-\infty}^\infty S_\epsilon[v_{\epsilon,\textbf{t}}]\partial_{t_i}\partial_{t_j}v_{\epsilon,\textbf{t}}e^{-\beta t}\,dt
 -\int\limits_{-\infty}^\infty \partial_{t_i}S_\epsilon[v_{\epsilon,\textbf{t}}]\partial_{t_j}v_{\epsilon,\textbf{t}}e^{-\beta t}\,dt.
\end{array}
\end{align*}
By (\ref{eq-v-phi}) we get
\begin{eqnarray*}
\partial_{t_i}S_\epsilon[v_{\epsilon,\textbf{t}}]=\sum_{k=1}^2c_k(\textbf{t})\partial_{t_i}Z_{\epsilon,t_{k}}+\sum_{k=1}^2\partial_{t_i}c_k(\textbf{t})Z_{\epsilon,t_{k}}
\end{eqnarray*}
Let $\textbf{t}_\epsilon$ be a critical point of $K_\epsilon(\textbf{t})$ in $\Lambda$, then
\begin{align*}
S_\epsilon[v_{\epsilon,\textbf{t}_\epsilon}]=0\quad\mathrm{and}\ c_k(\textbf{t}_\epsilon)=0,
\end{align*}
which implies
\begin{align*}
\partial_{t_i}S_\epsilon[v_{\epsilon,\textbf{t}}]\Big|_{\textbf{t}=\textbf{t}_\epsilon}
=\sum_{k=1}^2\partial_{t_i}c_k(\textbf{t}_\epsilon)Z_{\epsilon,t_{\epsilon,k}}.
\end{align*}
Note that
\begin{eqnarray}\label{ck-1}
\partial_{t_i}S_\epsilon[v_{\epsilon,\textbf{t}}]=L_\epsilon[\partial_{t_i}v_{\epsilon,\textbf{t}}]+p\big[|v_{\epsilon,\textbf{t}}|^{p-1}-|w_{\epsilon,\textbf{t}}|^{p-1}\big]\partial_{t_i}v_{\epsilon,\textbf{t}}=:\overline{L}_\epsilon[\partial_{t_i}v_{\epsilon,\textbf{t}}].
\end{eqnarray}

\medskip

As in Lemma \ref{pop3.2}, multiplying (\ref{ck-1}) by $\partial_{t_j}w_{j,t_j}$ and integrating by parts, we get $\partial_{t_i}c_k(\textbf{t}_\epsilon)=O(\beta)$. Hence
\begin{align*}
\int\limits_{-\infty}^\infty\partial_{t_i}S_\epsilon[v_{\epsilon,\textbf{t}}]
 \partial_{t_j}\phi_{\epsilon,\textbf{t}}e^{-\beta t}\,dt\Big|_{\textbf{t}=\textbf{t}_\epsilon}
 &=\sum_{k=1}^2\partial_{t_i}c_k(\textbf{t}_\epsilon)\int\limits_{-\infty}^\infty
  Z_{\epsilon,t_{\epsilon,k}}(\partial_{t_j}\phi_{\epsilon,\textbf{t}_\epsilon})e^{-\beta t}\,dt\nonumber\\
 &=-\sum_{k=1}^2\partial_{t_i}c_k(\textbf{t}_\epsilon)\int\limits_{-\infty}^\infty
  (\partial_{t_j}Z_{\epsilon,t_{\epsilon,k}})\phi_{\epsilon,\textbf{t}_\epsilon}e^{-\beta t}\,dt=o(\beta),
\end{align*}
and then
\begin{align*}
\partial_{t_i}\partial_{t_j}K_\epsilon(\textbf{t}_\epsilon)
&=-\int\limits_{-\infty}^\infty\partial_{t_i}S_\epsilon[v_{\epsilon,\textbf{t}}]
 \partial_{t_j}v_{\epsilon,\textbf{t}}e^{-\beta t}\,dt\Big|_{\textbf{t}=\textbf{t}_\epsilon}\nonumber\\
&=-\int\limits_{-\infty}^\infty \overline{L}_\epsilon[\partial_{t_i}w_{\epsilon,\textbf{t}}
 +\partial_{t_i}\phi_{\epsilon,\textbf{t}}]
 \partial_{t_j}w_{\epsilon,\textbf{t}}e^{-\beta t}\,dt\Big|_{\textbf{t}=\textbf{t}_\epsilon}+o(\beta).
\end{align*}
Note that
\begin{align*}
\int\limits_{-\infty}^\infty \overline{L}_\epsilon[\partial_{t_i}\phi_{\epsilon,\textbf{t}}]
 \partial_{t_j}w_{\epsilon,\textbf{t}}e^{-\beta t}\,dt
 =\int\limits_{-\infty}^\infty \partial_{t_i}\phi_{\epsilon,\textbf{t}}
\overline{L}_\epsilon[\partial_{t_j}w_{\epsilon,\textbf{t}}]
 e^{-\beta t}\,dt=o(\beta),
\end{align*}
since
\begin{align*}
\overline{L}_\epsilon[\partial_{t_j}w_{\epsilon,\textbf{t}}]
 =-Z_{\epsilon,t_j}+p|v_{\epsilon,\textbf{t}}|^{p-1}\partial_{t_j}w_{\epsilon,\textbf{t}}=O(\beta^{\tau}).
\end{align*}
Therefore,
\begin{align*}
\partial_{t_i}\partial_{t_j}K_\epsilon(\textbf{t}_\epsilon)
=-\int\limits_{-\infty}^\infty \overline{L}_\epsilon[\partial_{t_i}w_{\epsilon,\textbf{t}}]
 \partial_{t_j}w_{\epsilon,\textbf{t}}e^{-\beta t}\,dt\Big|_{\textbf{t}=\textbf{t}_\epsilon}+o(\beta).
\end{align*}
Using the following important estimate:
\begin{eqnarray}\label{w-L-w}
\int\limits_{-\infty}^\infty \overline{L}_\epsilon[\partial_{t_i}w_{\epsilon,\textbf{t}}]
 \partial_{t_j}w_{\epsilon,\textbf{t}}e^{-\beta t}\,dt
 =\left\{\begin{array}{lll}
-\frac{1}{4}e^{(t_1-t_2)/2}A_{\epsilon,3}\int\limits_{-\infty}^\infty w^pe^{t/2}\,dt+o(\beta),\quad\mathrm{for}\ i=j=1;\\\\
\frac{1}{4}e^{(t_1-t_2)/2}A_{\epsilon,3}\int\limits_{-\infty}^\infty w^pe^{t/2}\,dt+o(\beta),\quad\mathrm{for}\ i\neq j;\\\\
-\Big[\frac{1}{4}e^{(t_1-t_2)/2}+\frac{1}{2}e^{t_2}\Big] A_{\epsilon,3}\int\limits_{-\infty}^\infty w^pe^{t/2}\,dt+o(\beta),\quad\mathrm{for}\ i=j=2,
\end{array}
\right.
\end{eqnarray}
which will be proved in Appendix C, we get the desired result.

\qed

\section{The non-degeneracy result and eigenvalue estimates}

In this section we shall study the eigenvalue estimates for
\begin{eqnarray}\label{eq-l-u-h}
L_\epsilon(\phi):=\Delta\phi-\phi+p|u_\epsilon|^{p-1}\phi
\end{eqnarray}
and prove Theorem \ref{thm1.2}.

\medskip

\noindent{\bf Proof of Theorem \ref{thm1.2}.} Let $\lambda_k, e_k(\theta)$ with $\theta\in S^{N-1}$ be the eigenvalues and eigenfunctions of the Laplace-Beltrami operator on $S^{N-1}$. Then
\begin{equation*}
\lambda_0=0<\lambda_1=\cdots=\lambda_N=N-1<\lambda_{N+1}\leq\cdots,
\end{equation*}
and $e_k$ are normalized so that they form a complete orthonormal basis of $L^2(S^{N-1})$. In fact the set of eigenvalues is given by  $\{j(N-2+j)\ |\ j\geq0\}$.

\medskip

Suppose $\phi$ satisfies
\begin{equation*}
L_\epsilon (\phi)=0 \mbox{ in } R^N, \ \phi(x)\to 0 \mbox{ as }|x|\to \infty.
\end{equation*}
Put
\begin{equation*}
\phi_k(r)=\int_{S^{N-1}}\phi(r,\theta)e_k(\theta)d\theta,
\end{equation*}
then $\phi_k(r)\to 0$ as $r\to \infty$, and it satisfies
\begin{eqnarray}\label{sh-u}
\phi_k''+\frac{N-1}{r}\phi_k'-\phi_k+p|u_\epsilon|^{p-1}\phi_k+\frac{(-\lambda_k)}{r^2}\phi_k=0\quad \mathrm{in}\ (0,\infty)\quad\mathrm{and}\ \lim\limits_{r\rightarrow\infty}\phi_k(r)=0,
\end{eqnarray}
for $k=0,1,\cdots$.  We claim that $\phi_k=0$ for $k\geq N+1$.

\medskip

To this end, let us consider the eigenvalues of the problem
\begin{eqnarray}\label{ev-u}
\phi_k''+\frac{N-1}{r}\phi_k'-\phi_k+p|u_\epsilon|^{p-1}\phi_k+\frac{\nu}{r^2}\phi_k=0\quad \mathrm{in}\ (0,\infty)\quad\mathrm{and}\ \lim\limits_{r\rightarrow\infty}\phi_k(r)=0.
\end{eqnarray}

\medskip

The $l$-th eigenvalue of (\ref{ev-u}) can be characterized variationally as
\begin{eqnarray}\label{nu-u}
\nu_l(p)=\max\limits_{\dim(V)<l}\inf\limits_{\phi\in V^\perp}\frac{\int_0^\infty\big[|\phi'|^2+|\phi|^2\big]r^{N-1}\,dr-p\int_0^\infty|u_\epsilon|^{p-1}|\phi|^2r^{N-1}\,dr}{\int_0^\infty|\phi|^2r^{N-3}\,dr},
\end{eqnarray}
where $V$ runs through subspaces of $H_{r}^1(\mathbb{R}^{N})$ and $V^\perp$ is the set of $\phi\in H_{0,r}^1(\mathbb{R}^{N})$ satisfying $\int_0^\infty \phi ur^{N-3}=0$ for all $u\in V$, and $H_{r}^1(\mathbb{R}^{N})$ be the space of radial functions in $H^1(\mathbb{R}^{N})$. Thanks to Hardy's inequality:
\begin{eqnarray*}
\frac{(N-2)^2}{4}\int_0^\infty|\phi|^2r^{N-3}\,dr\leq\int_0^\infty|\phi'|^2r^{N-1}\,dr,
\end{eqnarray*}
the eigenvalues $\nu_1(p)\leq\nu_2(p)\leq\cdots$ are well defined. Using Hardy's embedding and a simple compactness argument involving the fast decay of $|u_\epsilon|^{p-1}$, there is an extremal for $\nu_l(p)$ which represents a solution to problem (\ref{ev-u}) for $\nu=\nu_l(p)$.

\medskip

To prove Theorem \ref{thm1.2} we need to know whether and when $\nu_l(p)$ equals $-\lambda_k$. To show this more information about solutions is required. So we consider the corresponding problems for $v_\epsilon$ using the Emden-Fowler transformation. Then the eigenvalue problem (\ref{ev-u}) becomes
\begin{eqnarray}\label{ev-v}
\tilde{L}_\epsilon[\psi]:=\psi''-\beta\psi'-(\gamma+e^{2t})\psi+p|v_\epsilon|^{p-1}\psi=-\nu\psi\ \mathrm{in}\ (-\infty,\infty)\quad \mathrm{and}\ \lim\limits_{|t|\rightarrow\infty}\psi(t)=0.
\end{eqnarray}

\medskip

For the proof of Theorem \ref{thm1.2}, let us consider first the radial mode $k=0$, namely $\lambda_k=0$. The following result, which contains elements of independent interest, gives the small eigenvalue estimates of $L_\epsilon$ and shows that $\psi_k=0$ for the mode $k=0$.

\begin{pop}\label{pop-eigen}
For $\epsilon$ small enough, the eigenvalue problem
\begin{eqnarray*}
L_\epsilon\phi_\epsilon=\mu_\epsilon\phi_\epsilon
\end{eqnarray*}
has exactly two small eigenvalues $\mu_\epsilon^j$, $j=1,2$, which satisfy
\begin{eqnarray}\label{eigen-j}
\frac{\mu_\epsilon^j}{\epsilon}\rightarrow -c_0\xi_j,\quad\mathrm{up\ to\ a\ subsequence\ as}\ \epsilon\rightarrow0,\ \mathrm{for}\ j=1,2,
\end{eqnarray}
where $\xi_j$'s are the eigenvalues of the Hessian matrix $\nabla^2\widetilde{K}_\epsilon$ and $c_0$ is a positive constant. Furthermore, the corresponding eigenfunctions $\phi_\epsilon^j$'s satisfy
\begin{eqnarray*}
\phi_\epsilon^j=\sum\limits_{i=1}^2\big[a_{ij}+o(1)\big]\partial_{t_i}w_{\epsilon,\textbf{t}}+O(\epsilon),\quad j=1,2,
\end{eqnarray*}
where $\textbf{a}_j=(a_{1,j},\dots,a_{2,j})^T$ is the eigenvector associated with $\xi_j$, namely,
\begin{eqnarray*}
\nabla^2\widetilde{K}_\epsilon \textbf{a}_j=\xi_j \textbf{a}_j.
\end{eqnarray*}
\end{pop}

\begin{rk}
By (\ref{eigen-j}) we know that $\mu_\epsilon\neq0$ and then obtain the non-degeneracy of $v_\epsilon$ in the space of $H^1$-radial symmetric functions.
\end{rk}

\medskip
\noindent
{\bf Proof of proposition \ref{pop-eigen}.}
To prove this Proposition, one may follow the arguments given in Section 5 of \cite{W} or Section 2 of \cite{LWY} and the following estimates
\begin{eqnarray}
\int\limits_{-\infty}^\infty \overline{L}_\epsilon[\partial_{t_i}w_{\epsilon,\textbf{t}}]
 \partial_{t_j}w_{\epsilon,\textbf{t}}e^{-\beta t}\,dt
 =\left\{\begin{array}{lll}
-\frac{1}{4}e^{(t_1-t_2)/2}A_{\epsilon,3}\int\limits_{-\infty}^\infty w^pe^{t/2}\,dt+o(\beta),\quad\mathrm{for}\ i=j=1;\\\\
\frac{1}{4}e^{(t_1-t_2)/2}A_{\epsilon,3}\int\limits_{-\infty}^\infty w^pe^{t/2}\,dt+o(\beta),\quad\mathrm{for}\ i\neq j;\\\\
-\Big[\frac{1}{4}e^{(t_1-t_2)/2}+\frac{1}{2}e^{t_2}\Big] A_{\epsilon,3}\int\limits_{-\infty}^\infty w^pe^{t/2}\,dt+o(\beta),\quad\mathrm{for}\ i=j=2,
\end{array}
\right.
\end{eqnarray}
given in Appendix C.

\qed

\medskip

Let us consider now mode $1$ for (\ref{sh-u}), namely $k=1,\dots,N$, for which $\lambda_k=N-1$. In this case we have an explicit solution $u_\epsilon'(r)$. Now we show that $\phi_k=C_k u_\epsilon'$ for some constants $C_k$ for $k=1,\dots,N$. This is not trivial since $u_\epsilon'(r)$ changes sign once. Suppose that $\phi_k$ solve (\ref{sh-u}). We first multiply equation of $\phi_k$ by $u_\epsilon'$ and the equation of $u_\epsilon'$ by $\phi_k$, and integrate over the ball $B_r$ centered at the origin with radius $r$. Since they satisfy the same equation, we get
\begin{eqnarray*}
\phi_k'(r)u_\epsilon'(r)-\phi_k(r)u_\epsilon''(r)=0,
\end{eqnarray*}
from which we get $\phi_k=C_k u_\epsilon'$ for some constants $C_k$.

\medskip

Finally let us consider modes $2$ and higher. Assume now that $k\geq N+1$ for which $\lambda_k\geq 2N$. Since $u_\epsilon'(r)$ has exactly one zero in $(0,\infty)$ and $\lambda_k>\lambda_1$, by the standard Sturm-Liouville comparison theorem, $\phi_k$ does not change sign in $(0,\infty)$. On the other hand, by Sturm-Liouville theory, it is well known that the eigenfunctions corresponding to $\nu_l$ much change sign in $(0,\infty)$ at least $l-1$ times. Thus the only possibility for equation (\ref{sh-u}) to have a nontrivial solution for a given $k\geq N+1$ is that $\lambda_k=-\nu_1(p)$. In the next proposition we shall show that $-\nu_1(p)\rightarrow\lambda_1=N-1$ as $p\rightarrow\frac{N+2}{N-2}$. Therefore we get $\lambda_k\neq-\nu_1(p)$ for $k\geq N+1$ when $p$ is closed to $\frac{N+2}{N-2}$ and then complete the proof of Theorem \ref{thm1.2}.

\qed

\begin{pop}\label{mode2}
As $p\uparrow\frac{N+2}{N-2}$, we have that $-\nu_l(p)\rightarrow\lambda_1=N-1$ for $l\leq2$.
\end{pop}

\medskip
\noindent
{\bf Proof of proposition \ref{mode2}.}
One may follow the arguments in Section 3 of \cite{DW}. Note that by the Emden-Fowler transformation, the eigenvalues have a variational characterization
\begin{eqnarray}\label{nu-v}
\nu_l(p)=\max\limits_{\dim(W)<l}\inf\limits_{\psi\in W^\perp}\frac{\int_{-\infty}^\infty\big[|\psi'|^2+(\gamma+e^{2t})|\psi|^2\big]e^{-\beta t}\,dt-p\int_{-\infty}^\infty|v_\epsilon|^{p-1}|\psi|^2e^{-\beta t}\,dt}{\int_{-\infty}^\infty|\psi|^2e^{-\beta t}\,dt},
\end{eqnarray}
where $W$ runs through the subspaces of $H$ and $W^\perp$ is the set of $\psi\in W$ satisfying $\int_{-\infty}^\infty \psi ve^{-\beta t}\,dt=0$ for all $v\in W$. Note that the term involving the weight is relatively compact and it follows from a previous argument that the eigenvalues exist.

\medskip

Observe that the limiting eigenvalue problem
\begin{eqnarray*}
\psi''-\frac{(N-2)^2}{4}\psi+\frac{N+2}{N-2}w_0^{\frac{4}{N-2}}\psi=\mu\psi,\quad\psi(\pm\infty)=0,
\end{eqnarray*}
admits eigenvalues
\begin{eqnarray}\label{mu123}
\mu_1=N-1,\quad\mu_2=0,\quad\mu_3<0,\quad \cdots,
\end{eqnarray}
where the corresponding eigenfunction for the principal eigenvalue $\mu_1$ is positive and denoted by $\Psi_1$. A simple computation shows that we can take $\Psi_1=w_0^{\frac{N}{N-2}}$. Now we take $\psi_j=w_{j,\textbf{t}_{\epsilon,j}}^{\frac{p+1}{2}}$, $j=1,2$. Let $W$ be a given one-dimensional subspace. Then there exists $c_1,c_2$ (not all equal to $0$) such that $\int_{-\infty}^\infty\big(\sum\limits_{j=1}^2c_j\psi_j\big)ve^{-\beta t}\,dt=0$ for all $v\in W$. We then compute that
\begin{eqnarray*}
&\int_{-\infty}^\infty\big[|\psi'|^2+(\gamma+e^{2t})|\psi|^2\big]e^{-\beta t}\,dt-p\int_{-\infty}^\infty|v_\epsilon|^{p-1}|\psi|^2e^{-\beta t}\,dt\\
&\leq\sum\limits_{j=1}^2c_j^2\big(-\mu_1+o(1)\big)\int_{-\infty}^\infty|\psi|^2e^{-\beta t}\,dt,
\end{eqnarray*}
and hence by variational characterization of $\nu_2$ we deduce that
\begin{eqnarray}
\nu_l(p)\leq \nu_2(p)\leq-(N-1)+o(1),\quad l=1,2.
\end{eqnarray}
On the other hand, according to (\ref{mu123}), $v_l(p)\rightarrow\mu_k\geq-(N-1)$ for some $k$. Thus we have $\nu_l(p)\rightarrow-(N-1)$ as $p\rightarrow\frac{N+2}{N-2}$ for $l\leq2$.

\qed

\section{Appendices}

\subsection{Appendix A}

In this subsection we shall give the estimates of $w_{j,t_j}$, $j=1,2$. Recall that $w_{j,t_j}$ is the unique solution to the following equation
\begin{align}\label{eq-w-j}
v''-(\gamma_0+e^{2s})v+w_{t_j}^p=0,\ v\in H
\end{align}
whose existence is given by the Riesz's representation Theorem. Here $w$ is the unique positive even solution of
\begin{align}\label{eq-w-0}
w''-\gamma_0w+w^p=0.
\end{align}
In fact, the function $w(t)$ can be written explicitly and has the
following form
\begin{align*}
w(t)=\gamma_0^{\frac{1}{p-1}}(\frac{p+1}{2})^\frac{1}{p-1}
\Big[\cosh\big(\frac{p-1}{2}\gamma_0^{1/2}t\big)\Big]^{-\frac{2}{p-1}}
=A_{\epsilon,N}\Big[e^{\frac{p-1}{2}\gamma_0^{1/2}t}
+e^{-\frac{p-1}{2}\gamma_0^{1/2}t}\Big]^{-\frac{2}{p-1}}.
\end{align*}
Note that now $w$ has the following expansion
\begin{align*}
\left\{\begin{array}{ll}
w(t)=A_{\epsilon,N}e^{-\sqrt{\gamma_0}t}+O(e^{-p\sqrt{\gamma_0}t}),\quad t\geq0;\\\\
w'(t)=-\sqrt{\gamma_0}A_{\epsilon,N}e^{-\sqrt{\gamma_0}t}+O(e^{-p\sqrt{\gamma_0}t}),\quad
t\geq0,
\end{array}
\right.
\end{align*}
where $A_{\epsilon,N}>0$ is a constant depending on $\epsilon$ and $N$.

\medskip

To get the estimates of $w_{j,t_j}$, we write $w_{j,t_j}=w_{t_j}+\phi$, then by (\ref{eq-w-j}) and (\ref{eq-w-0}), $\phi$ satisfies
\begin{eqnarray}\label{eq-phi-j}
\phi''-(\gamma_0+e^{2s})\phi-e^{2s}w_{t_j}=0.
\end{eqnarray}

\medskip

Note that as $s\rightarrow\infty$,
$e^{2s}w_{t_j}(s)\to e^{\frac{N-2}{2}t_j}A_{\epsilon,N}e^{-\frac{N-6}{2}s}$. Hence when $N>6$, $\phi\in H$ and $\phi=O(e^{2t_j})$ . Therefore,
\begin{eqnarray}\label{N>6}
w_{j,t_j}=w_{t_j}+O(e^{2t_j}),\quad\mathrm{when}\ N>6.
\end{eqnarray}

\medskip

Next we consider $N\leq6$, let $\phi_N$ be the unique solution of
\begin{eqnarray}\label{eq-phi-t}
\phi''-(\gamma_0+e^{2s})\phi-e^{-\frac{N-6}{2}s}=0,\quad|\phi(s)|\rightarrow0,\ \mathrm{as}\ |s|\rightarrow\infty,
\end{eqnarray}
then
\begin{eqnarray}\label{N<6}
w_{j,t_j}=w_{t_j}+e^{\frac{N-2}{2}t_j}A_{\epsilon,N}\phi_N+O(e^{2t_j})=:w_{t_j}+\phi_{j,t_j}+O(e^{2t_j}),\quad\mathrm{when}\ N\leq6.
\end{eqnarray}

\medskip

The rest of this subsection will be devoted to the solvability of $\phi_N$. A key observation is that
\begin{align}
\phi_0=-e^{-\frac{N-2}{2}s}
\end{align}
is a special solution of (\ref{eq-phi-t}). Thus if we write
\begin{equation*}
\phi_N=\phi_0+\phi,
\end{equation*}
in order to find a solution of (\ref{eq-phi-t}) which satisfies the decay condition at $\infty$,    let
\begin{align}
\phi(s)=e^{-\frac{N-2}{2}s}\widetilde{\phi}\big(\lambda_N
e^{(N-2)s}\big),\quad\mathrm{where}\ \lambda_N=(N-2)^{-(N-2)}.
\end{align}
Then $\widetilde{\phi}$ satisfies
\begin{align}\label{phi-N}
\widetilde{\phi}''(s)=s^{-\frac{2N-6}{N-2}}\widetilde{\phi}(s),\quad\widetilde{\phi}(0)=1,
\quad\widetilde{\phi}(\infty)=0
\end{align}
and thus
\begin{align*}
\phi_N=-e^{-\frac{N-2}{2}s}\Big[1-\widetilde{\phi}\big(\lambda_N
e^{(N-2)s}\big)\Big].
\end{align*}

In the case of $N=3$, $\lambda_3=1$ and $\widetilde{\phi}=e^{-s}$. Then
\begin{align*}
\phi_3=-e^{-s/2}\big(1-e^{-e^s}\big).
\end{align*}
In the case of $N=4$, $\lambda_4=1/4$ and
\begin{eqnarray*}
\widetilde{\phi}(r)=2\sqrt{r}K_1(2\sqrt{r})=:\rho_0,
\end{eqnarray*}
where $K_1(z)$ is the modified Bessel function of second kind and
satisfies
\begin{eqnarray*}
z^2K_1''(z)+zK_1'(z)-(z^2+1)K_1(z)=0,
\end{eqnarray*}
see for example \cite{MO}. Then
\begin{align*}
\phi_4=-e^{-s}\Big[1-\rho_0(\frac{1}{4}e^{2s})\Big].
\end{align*}
For $N=5$,
\begin{align*}
\phi_5=-e^{-3s/2}\Big[1-(1+e^s)
 e^{-e^s}\Big].
\end{align*}
%
%
In the case of $N=6$,
\begin{align*}
\phi_6=-e^{-2s}\big[1-u_0(\frac{1}{16^2}e^{4s})\big],
\end{align*}
where $u_0$ satisfies
\begin{align*}
u''(r)=\frac{u(r)}{r^{3/2}},\ u(0)=1, u(\infty)=0.
\end{align*}
Actually, we have
\begin{align*}
u_0(r)=8\sqrt{r}K_2(4r^{1/4}),
\end{align*}
where $K_2(z)$ is the modified Bessel function of second kind and
satisfies
\begin{align*}
z^2K_2''(z)+zK_2'(z)-(z^2+4)K_2(z)=0.
\end{align*}

\subsection{Appendix B}

In this appendix we expand the quality
$E_\epsilon[w_{\epsilon,\textbf{t}}]$ in terms of $\epsilon$
and $\textbf{t}$.

\begin{lem}\label{E0}
For $\textbf{t}\in\Lambda$ and $\epsilon$ sufficiently small, we
have for $N=3$,
\begin{align*}
E_\epsilon[w_{\epsilon,\textbf{t}}]&=
\big(\frac{1}{2}-\frac{1}{p+1}\big)(e^{-\beta t_1}+e^{-\beta
t_2})\int_{-\infty}^\infty
w^{p+1}\,dt+\frac{1}{2}e^{t_2}A_{\epsilon,3}\int_{-\infty}^\infty w^{p}e^{t/2}\,dt\nonumber\\
 &\quad+e^{-|t_1-t_2|/2}A_{\epsilon,3}\int_{-\infty}^\infty
w^pe^{t/2}\,dt +o(\beta)+o(e^{t_2})+o(e^{-|t_1-t_2|/2}).
\end{align*}

For $N=4$,
\begin{align*}
E_\epsilon[w_{\epsilon,\textbf{t}}]&=
\big(\frac{1}{2}-\frac{1}{p+1}\big)(e^{-\beta t_1}+e^{-\beta
t_2})\int_{-\infty}^\infty
w^{p+1}\,dt-\frac{1}{4}t_2e^{2t_2}A_{\epsilon,4}\int_{-\infty}^\infty w^{p}e^{t}\,dt\nonumber\\
 &\quad+e^{-|t_1-t_2|}A_{\epsilon,4}\int_{-\infty}^\infty w^pe^{t}\,dt
  +o(\beta)+o(t_2e^{2t_2})+o(e^{-|t_1-t_2|}).
\end{align*}

For $N\geq5$,
\begin{align*}
E_\epsilon[w_{\epsilon,\textbf{t}}]&=
\big(\frac{1}{2}-\frac{1}{p+1}\big)(e^{-\beta t_1}+e^{-\beta
t_2})\int_{-\infty}^\infty
w^{p+1}\,dt+\frac{1}{2}e^{2t_2}\int_{-\infty}^\infty w^{2}e^{2t}\,dt\nonumber\\
 &\quad+e^{-(N-2)|t_1-t_2|/2}A_{\epsilon,N}\int_{-\infty}^\infty
  w^pe^{(N-2)t/2}\,dt +o(\beta)+o(e^{2t_2})+o(e^{-(N-2)|t_1-t_2|/2}).
\end{align*}

\end{lem}

\begin{proof}
Since the proofs are similar for different cases, we give the details for $N=3$ here. Integrating by parts we get
\begin{align*}
E_\epsilon[w_{\epsilon,\textbf{t}}]
&=\frac{1}{2}\int_{-\infty}^\infty\Big[
 -S_\epsilon[w_{\epsilon,\textbf{t}}]+|w_{\epsilon,\textbf{t}}|^{p-1}w_{\epsilon,\textbf{t}}\Big]
 w_{\epsilon,\textbf{t}}e^{-\beta
 t}\,dt-\frac{1}{p+1}\int_{-\infty}^\infty|w_{\epsilon,\textbf{t}}|^{p+1}
 e^{-\beta t}\,dt\\
&=\frac{1}{2}\int_{-\infty}^\infty\Big[ \beta
 w_{\epsilon,\textbf{t}}'+(\gamma-\gamma_0)w_{\epsilon,\textbf{t}}+w_{t_1}^p-w_{t_2}^p\Big]
 w_{\epsilon,\textbf{t}}e^{-\beta
 t}\,dt-\frac{1}{p+1}\int_{-\infty}^\infty|w_{\epsilon,\textbf{t}}|^{p+1}
 e^{-\beta t}\,dt\\
&=E_1+E_2+E_3-E_4+E_5,
\end{align*}
where
\begin{align*}
E_1=\frac{\beta}{2}\int_{-\infty}^\infty
w_{\epsilon,\textbf{t}}'w_{\epsilon,\textbf{t}}e^{-\beta t}\,dt
=\frac{\beta^2}{4}\int_{-\infty}^\infty
w_{\epsilon,\textbf{t}}^2e^{-\beta t}\,dt=O(\beta^2);
\end{align*}
\begin{align*}
E_2=\frac{(\gamma-\gamma_0)}{2}\int_{-\infty}^\infty
w_{\epsilon,\textbf{t}}^2e^{-\beta t}\,dt
=-\frac{\beta^2}{8}\int_{-\infty}^\infty
w_{\epsilon,\textbf{t}}^2e^{-\beta t}\,dt=O(\beta^2);
\end{align*}
\begin{align*}
E_3=-\frac{1}{2}\int_{-\infty}^\infty w_{t_1}^{p}w_{2,t_2}e^{-\beta
t}\,dt-\frac{1}{2}\int_{-\infty}^\infty
w_{1,t_1}w_{t_2}^{p}e^{-\beta t}\,dt;
\end{align*}
\begin{align*}
E_4=\frac{1}{p+1}\int_{-\infty}^\infty
\Big[|w_{1,t_1}-w_{2,t_2}|^{p+1}-w_{t_1}^{p}w_{1,t_1}-w_{t_2}^{p}w_{2,t_2}\Big]e^{-\beta
t}\,dt;
\end{align*}
\begin{align*}
E_5=\big(\frac{1}{2}-\frac{1}{p+1}\big)\Big[\int_{-\infty}^\infty
w_{t_1}^{p}w_{1,t_1}e^{-\beta t}\,dt+\int_{-\infty}^\infty
w_{t_2}^{p}w_{2,t_2}e^{-\beta t}\,dt\Big].
\end{align*}

\medskip

First for $E_3$, by Lemma \ref{lem-B} we have
\begin{align*}
E_3=-e^{-|t_1-t_2|/2}A_{\epsilon,3}\int_{-\infty}^\infty
 w^pe^{t/2}\,dt+o(\beta)+o(e^{t_2})+o(e^{-|t_1-t_2|/2}).
\end{align*}

To estimate $E_4$, we divide $\mathbb{R}$ into two intervals
$I_1,I_2$ defined by
$$I_1=(-\infty,\frac{t_1+t_2}{2}),\quad I_2=[\frac{t_1+t_2}{2},\infty).$$
So on $I_1$ the following equality holds:
\begin{eqnarray*}
&&\frac{1}{p+1}\Big[|w_{1,t_1}-w_{2,t_2}|^{p+1}
 -w_{t_1}^{p}w_{1,t_1}-w_{t_2}^{p}w_{2,t_2}\Big]\\
&=&\frac{1}{p+1}\big[(w_{1,t_1}-w_{2,t_2})^{p+1}-w_{1,t_1}^{p+1}+(p+1)w_{1,t_1}^{p}w_{2,t_2}\big]-w_{1,t_1}^{p}w_{2,t_2}\\
&&+\frac{1}{p+1}\big[(w_{t_1}+\phi_{1,t_1})^p-w_{t_1}^p-pw_{t_1}^{p-1}\phi_{1,t_1}\big]w_{1,t_1}+\frac{p}{p+1}w_{t_1}^{p}\phi_{1,t_1}\\
&&+\frac{p}{p+1}w_{t_1}^{p-1}\phi_{1,t_1}^2-\frac{1}{p+1}w_{t_2}^pw_{2,t_2}.
\end{eqnarray*}
As in the proof of Lemma \ref{basic lemma}, by the mean value theorem and inequality (\ref{ieq-1}) we have
\begin{eqnarray*}
&&\Big|\frac{1}{p+1}\Big[|w_{1,t_1}-w_{2,t_2}|^{p+1}
 -w_{t_1}^{p}w_{1,t_1}-w_{t_2}^{p}w_{2,t_2}\Big]\\
 &&+w_{1,t_1}^pw_{2,t_2}-\frac{p}{p+1}w_{t_1}^p\phi_{1,t_1}\Big|
\leq Cw_{t_1}^{p+1-\delta}w_{t_2}^\delta,
\end{eqnarray*}
for any $1<\delta<2$.

\medskip

Using Lemma \ref{lem2.1} and integrating by parts, we get
\begin{align*}
&\quad\frac{1}{p+1}\int_{I_1}
 \Big[|w_{t_1}-w_{t_2}|^{p+1}-w_{t_1}^{p}w_{1,t_1}-w_{t_2}^{p}w_{2,t_2}\Big]e^{-\beta
 t}\,dt\\
&=-\frac{p}{p+1}e^{t_1}A_{\epsilon,3}\int_{-\infty}^\infty
 w^{p}e^{t/2}\,dt
 -e^{-|t_1-t_2|/2}A_{\epsilon,3}\int_{-\infty}^\infty
 w^pe^{t/2}\,dt+o(e^{-|t_1-t_2|/2}).
\end{align*}
Similarly,
\begin{align*}
&\quad\frac{1}{p+1}\int_{I_2}
 \Big[|w_{t_1}-w_{t_2}|^{p+1}-w_{t_1}^{p}w_{1,t_1}-w_{t_2}^{p}w_{2,t_2}\Big]e^{-\beta
 t}\,dt\\
&=-\frac{p}{p+1}e^{t_2}A_{\epsilon,3}\int_{-\infty}^\infty
 w^{p}e^{t/2}\,dt
 -e^{-|t_1-t_2|/2}A_{\epsilon,3}\int_{-\infty}^\infty
 w^pe^{t/2}\,dt+o(e^{-|t_1-t_2|/2}).
\end{align*}
Hence
\begin{align*}
E_4=-\frac{p}{p+1}e^{t_2}A_{\epsilon,3}\int_{-\infty}^\infty
w^{p}e^{t/2}\,dt
-2e^{-|t_1-t_2|/2}A_{\epsilon,3}\int_{-\infty}^\infty
w^pe^{t/2}\,dt+o(e^{-|t_1-t_2|/2}).
\end{align*}

\medskip

Regarding the term $E_5$, by the Lemma \ref{lem-B} we have
\begin{align*}
E_5&=\big(\frac{1}{2}-\frac{1}{p+1}\big)\Big[\int_{-\infty}^\infty w_{t_1}^{p+1}e^{-\beta t}\,dt+\int_{-\infty}^\infty w_{t_1}^p\phi_{1,t_1}e^{-\beta t}\,dt\\
&+\int_{-\infty}^\infty w_{t_2}^{p+1}e^{-\beta t}\,dt+\int_{-\infty}^\infty w_{t_2}^p\phi_{2,t_2}e^{-\beta t}\,dt\Big]\\
&=\big(\frac{1}{2}-\frac{1}{p+1}\big)(e^{-\beta t_1}+e^{-\beta
 t_2})\int_{-\infty}^\infty
 w^{p+1}\,dt\\
 &-\big(\frac{1}{2}-\frac{1}{p+1}\big)e^{t_2}A_{\epsilon,3}\int_{-\infty}^\infty
 w^{p}e^{t/2}\,dt+o(\beta).
\end{align*}

Combining the above estimates for $E_1,E_2,E_3,E_4$ and $E_5$, we obtain
\begin{align*}
E_\epsilon[w_{\epsilon,\textbf{t}}]&=
 \big(\frac{1}{2}-\frac{1}{p+1}\big)(e^{-\beta t_1}+e^{-\beta
 t_2})\int_{-\infty}^\infty
 w^{p+1}\,dt+\frac{1}{2}e^{t_2}A_{\epsilon,3}\int_{-\infty}^\infty w^{p}e^{t/2}\,dt\nonumber\\
&\quad+e^{-|t_1-t_2|/2}A_{\epsilon,3}\int_{-\infty}^\infty
 w^pe^{t/2}\,dt +o(\beta)+o(e^{t_2})+o(e^{-|t_1-t_2|/2}).
\end{align*}

\end{proof}

\subsection{Appendix C}

In this section we give the technical proof of (\ref{w-L-w}) for $N=3$, that is,
\begin{eqnarray}\label{w-L-w-p}
\int\limits_{-\infty}^\infty \overline{L}_\epsilon[\partial_{t_i}w_{\epsilon,\textbf{t}}]
 \partial_{t_j}w_{\epsilon,\textbf{t}}e^{-\beta t}\,dt
 =\left\{\begin{array}{lll}
-\frac{1}{4}e^{(t_1-t_2)/2}A_{\epsilon,3}\int\limits_{-\infty}^\infty w^pe^{t/2}\,dt+o(\beta),\quad\mathrm{for}\ i=j=1;\\\\
\frac{1}{4}e^{(t_1-t_2)/2}A_{\epsilon,3}\int\limits_{-\infty}^\infty w^pe^{t/2}\,dt+o(\beta),\quad\mathrm{for}\ i\neq j;\\\\
-\Big[\frac{1}{4}e^{(t_1-t_2)/2}+\frac{1}{2}e^{t_2}\Big] A_{\epsilon,3}\int\limits_{-\infty}^\infty w^pe^{t/2}\,dt+o(\beta),\quad\mathrm{for}\ i=j=2.
\end{array}
\right.
\end{eqnarray}

\begin{proof}
Note that by (\ref{z-1-1}) and (\ref{Z}), we obtain
\begin{align}\label{C-1}
\overline{L}_\epsilon[\partial_{t_j}w_{\epsilon,\textbf{t}}]
=&-Z_{\epsilon,t_j}+p|v_{\epsilon,\textbf{t}}|^{p-1}\partial_{t_j}w_{\epsilon,\textbf{t}}\\
=&(-1)^{j}\Big[-pw_{t_j}^{p-1}w_{t_j}'
+\beta(\partial_{t_j}w_{j,t_j})'+(\gamma-\gamma_0)\partial_{t_j}w_{j,t_j}-p|v_{\epsilon,\textbf{t}}|^{p-1}\partial_{t_j}w_{j,t_j}\Big],\nonumber
\end{align}
and by the definition of $w_{\epsilon,\textbf{t}}$,
\begin{align}\label{C-2}
\partial_{t_j}w_{\epsilon,\textbf{t}}=(-1)^{j+1}\partial_{t_j}w_{j,t_j}=(-1)^{j+1}\big(\partial_{t_j}w_{t_j}+\partial_{t_j}\phi_{j,t_j}\big)+O(e^{2t_j}).
\end{align}

\medskip

In order to calculate the integration, we divide $(-\infty,\infty)$ into two intervals $I_1,I_2$
defined by
$$I_1=(-\infty,\frac{t_1+t_2}{2}),\quad I_2=[\frac{t_1+t_2}{2},\infty).$$

First we computer the case of $i\neq j$. By (\ref{C-1}) and (\ref{C-2}) we get
\begin{eqnarray*}
\int\limits_{-\infty}^\infty \overline{L}_\epsilon[\partial_{t_1}w_{\epsilon,\textbf{t}}]
\partial_{t_2}w_{\epsilon,\textbf{t}}e^{-\beta t}\,dt
&=&\int\limits_{-\infty}^\infty pw_{t_1}^{p-1}w_{t_1}'w_{t_2}'
-\int_{I_1}p|v_{\epsilon,\textbf{t}}|^{p-1}w_{t_1}'w_{t_2}'
-\int_{I_2}p|v_{\epsilon,\textbf{t}}|^{p-1}w_{t_1}'w_{t_2}'+o(\beta)\nonumber\\
&=&-\int\limits_{I_2} pw_{t_2}^{p-1}w_{t_1}'w_{t_2}'+o(\beta)\\
&=&\int\limits_{-\infty}^\infty w_{t_2}^pw_{t_1}''+o(\beta)\nonumber\\
&=&\frac{1}{4}e^{(t_1-t_2)/2}A_{\epsilon,3}\int\limits_{-\infty}^\infty w^pe^{t/2}\,dt+o(\beta).\nonumber
\end{eqnarray*}

For the case of $i=j=1$, recall that $v_{\epsilon,\textbf{t}}=w_{\epsilon,\textbf{t}}+\phi$, where $\phi=\phi_{\epsilon,\textbf{t}}$ is given by Proposition \ref{pop4.2}. Then on $I_1$:
\begin{eqnarray*}
&&p|v_{\epsilon,\textbf{t}}|^{p-1}(w_{t_1}')^2
-p|w_{t_1}|^{p-1}(w_{t_1}')^2\\
&=&
-p(p-1)w_{t_1}^{p-2}(w_{t_1}')^2w_{t_2}+p(p-1)w_{t_1}^{p-2}(w_{t_1}')^2\phi+o(\beta).
\end{eqnarray*}
So by (\ref{C-1}) and (\ref{C-2}) we obtain
\begin{eqnarray}\label{w-1-w-1}
&&\int\limits_{-\infty}^\infty \overline{L}_\epsilon[\partial_{t_1}w_{\epsilon,\textbf{t}}]
\partial_{t_1}w_{\epsilon,\textbf{t}}e^{-\beta t}\,dt\\
&=&-\int_{I_1} pw_{t_1}^{p-1}(w_{t_1}')^2e^{-\beta t}\,dt+\int_{I_1}p|v_{\epsilon,\textbf{t}}|^{p-1}(w_{t_1}')^2e^{-\beta t}\,dt+o(\beta)\nonumber\\
&=&-\int_{I_1}p(p-1)w_{t_1}^{p-2}(w_{t_1}')^2w_{t_2}\,dt+\int_{I_1}p(p-1)w_{t_1}^{p-2}(w_{t_1}')^2\phi+o(\beta)\nonumber\\
&=&T_1+T_2+o(\beta).\nonumber
\end{eqnarray}
Recall that $w_{j,t_j}$ satisfies
\begin{eqnarray*}
w_{j,t_j}''-(\gamma_0+e^{2t})w_{j,t_j}+w_{t_j}^p=0.
\end{eqnarray*}
So $\partial_{t_j}w_{j,t_j}$ and $\partial_{t_j}^2w_{j,t_j}$ satisfy
\begin{eqnarray*}
(\partial_{t_j}w_{j,t_j})''-(\gamma_0+e^{2t})(\partial_{t_j}w_{j,t_j})+pw_{t_j}^{p-1}(\partial_{t_j}w_{t_j})=0,
\end{eqnarray*}
and
\begin{eqnarray*}
(\partial_{t_j}^2w_{j,t_j})''-(\gamma_0+e^{2t})(\partial_{t_j}^2w_{j,t_j})+pw_{t_j}^{p-1}(\partial_{t_j}^2w_{t_j})+p(p-1)w_{t_j}^{p-2}(\partial_{t_j}w_{t_j})^2=0,
\end{eqnarray*}
which implies
\begin{eqnarray*}
p(p-1)w_{t_1}^{p-2}(w_{t_1}')^2=-L_\epsilon[\partial_{t_1}^2w_{1,t_1}]+o(\beta).
\end{eqnarray*}

\medskip

Hence
\begin{eqnarray*}
T_2=-\int_{I_1}\phi L_\epsilon[\partial_{t_1}^2w_{1,t_1}]+o(\beta).
\end{eqnarray*}
By (\ref{s[w]}) and Proposition \ref{pop4.2}, on $I_1$ we have on $I_1$
\begin{eqnarray*}
L_\epsilon[\phi]=\beta w_{t_1}'+pw_{t_1}^{p-1}w_{t_2}+o(\beta).
\end{eqnarray*}
Thus
\begin{eqnarray}\label{T2}
T_2&=-\int_{\mathbb{R}}w_{t_1}''\big[\beta w_{t_1}'+pw_{t_1}^{p-1}w_{t_2}\big]\,dt+o(\beta)\\
&=-\int_{\mathbb{R}}w_{t_1}''pw_{t_1}^{p-1}w_{t_2}\,dt+o(\beta).\nonumber
\end{eqnarray}
On the other hand,
\begin{eqnarray}\label{T1}
T_1&=&-\int_{\mathbb{R}}p(p-1)w_{t_1}^{p-2}(w_{t_1}')^2w_{t_2}\,dt+o(\beta)\nonumber\\
&=&\int_{\mathbb{R}}L_0[w_{t_1}'']w_{t_2}\,dt+o(\beta)\nonumber\\
&=&\int_{\mathbb{R}}w_{t_1}''L_0[w_{t_2}]\,dt+o(\beta)\nonumber\\
&=&\int_{\mathbb{R}}w_{t_1}''\big[w_{t_2}^p+pw_{t_1}^{p-1}w_{t_2}\big]\,dt+o(\beta)\nonumber\\
&=&-\int_{\mathbb{R}}w_{t_1}''w_{t_2}^p\,dt+\int_{\mathbb{R}}w_{t_1}''pw_{t_1}^{p-1}w_{t_2}\,dt+o(\beta),
\end{eqnarray}
where
\begin{eqnarray*}
L_0[\phi]:=\phi''-\gamma_0\phi+pw_{t_1}^{p-1}\phi.
\end{eqnarray*}
Combining (\ref{w-1-w-1}), (\ref{T2}) and (\ref{T1}), we get the desired result for $i=j=1$. The proof for $i=j=2$ is similar, we omit the details here.

\end{proof}



\begin{thebibliography}{99}
\bibitem{ABC} A.\ Ambrosetti, M.\ Badiale and S.\ Cingolani, {\em Semiclassical states of nonlinear
 Schr\"{o}dinger equations,} Arch. Rational Mech. Anal. 140 (1997), 285-300.

%
\bibitem{BL}
A. Bahri and Y. Li,
{\it On a minimax procedure for the existence of a positive solution for certain scaler field equation in $R^n$},
Revista Mat. Iberoa. 6 (1990), 1-15.

\bibitem{BW} T. Bartsch and M.  Willem, {\em Infinitely many radial solutions of a semilinear elliptic problem on $R^n$}, Arch. Rat. Mech. Anal. 124 (1993), 261–276.


 
\bibitem{bl1} H. Berestycki and P.L. Lions, {\em Nonlinear scalar field equations. I. Existence of a ground state},
Arch. Rational Mech. Anal., 82 (1983), 313-345.

\bibitem{bl2} H. Berestycki and P.L. Lions, {\em Nonlinear scalar field equations, II}, Arch. Rat. Mech. Anal. 82 (1981), 347-375.


\bibitem{CL} C.\ C.\ Chen, C.\ S.\ Lin,
{\em Uniqueness of the ground state solutions of $\Delta u+f(u)=0$
in $\mathbb{R}^N$, $N\geq3$}, Comm. Partial Differential Equations
16 (1991), 1549-1572.

\bibitem{C1} S.\ V.\ Coffman, {\em Uniqueness of the ground state solution of $\Delta u-u+u^3=0$ and
a variational characterization of other solutions}, Arch. Ration.
Mech. Anal. 46 (1972), 81-95.

\bibitem{C2} S.\ V.\ Coffman, {\em A nonlinear boundary value problem with many positive solutions},
J. Differential Equations 54 (1984), 429-437.


\bibitem{CGY} C.\ Cot$\acute{a}$zar, M.\ Garc$\acute{i}$a-Huidobro and C.\ S.\ Yarur, {\em
On the uniqueness of the second bound state solution of a semilinear
equation,} Ann. I. H. Poincar$\acute{e}$-AN 26 (2009), no. 6, 2091-2110.

\bibitem{conti} M.Conti, L. Verizzi,  and S. Terracini, {\em Radial solutions of superlinear problems in $R^N$}. Arch. Rat. Mech. Anal. 153 (1993), 291–316.

\bibitem{DW} E.\ N.\ Dancer, J.\ Wei,
{\em Sign-changing solutions for supercritical elliptic problems in domains with small holes,} Manuscripta Mathematica 123 (2007), no. 4, 493-511.


\bibitem{FQTY} P.\ L.\ Felmer, A.\ Quaas, M.\ Tang and J.\ Yu, {\em Monotonicity properties for ground states of the scalar field equation,}
Ann. I. H. Poincar$\acute{e}$-AN, 25 (2008), 105-119.

\bibitem{FW}A. Floer and A. Weinstein, {\em Nonspreading wave packets for the cubic
 Schr\"{o}dinger equation with a bounded potential,} J. Funct. Anal., 69 (1986), 397-408.

\bibitem{GNN}B.\ Gidas, W.-M.\ Ni and L.\ Nirenberg, {\em Symmetry of
positive solutions of nonlinear elliptic equations in
$\mathbb{R}^n$}, In Mathematical analysis and applications. Part A.
Adances in Mathematical Supplementary Studies, vol.7A, pp. 369-402
(Academic,1981).

\bibitem{GW}
 C.\ Gui and J.\ Wei,  {\em Multiple interior peak solutions for some singularly perturbed Neumann problems,} J. Differential Equations 158 (1999), no. 1, 1-27.

\bibitem{GWW}
 C.\ Gui, J.\ Wei, and M.\ Winter, {\em Multiple boundary peak solutions for some singularly perturbed Neumann problems,} Ann. I. H. Poincar$\acute{e}$ 17 (2000),
47-82.

\bibitem{G}M.\ Grossi, {\em On the number of single-peak solutions of the nonlinear Schr\"{o}dinger equations}, Ann. I. H. Poincare-AN 19, (2002), no.3, 261-280.

\bibitem{LWY}Tai-Chia.\ Lin, J.\ Wei and W.\ Yao, {\em Orbital stability of bound states of nonlinear Sch\"odinger equations with linear and nonlinear optical lattices,}  J. Differential Equation 249 (2010), 2111-2146.

\bibitem{KW}X.\ Kang and J.\ Wei, {\em On interacting bumps of semi-classical states of nonlinear
 Schr\"odinger equations,}  Adv. Diff. Eqns. 5 (2000), no.7-9, 899-928.

\bibitem{K}M.\ K.\ Kwong, {\em Uniqueness of positive solutions of
$\Delta u-u+u^p=0\in\mathbb{R}^n,$} Arch. Ration. Mech. Anal. 105
(1989), 243-266.

\bibitem{MNW} A.\ Malchiodi, W.-M.\ Ni and J.\ Wei,
{\em Multiple clustered layer solutions for semilinear Neumann
problems on a ball,} Ann. I. H. Poincar$\acute{e}$ 22 (2005),
143-163.

\bibitem{MS}K.\ McLeod and J.\ Serrin, {\em Uniqueness of positive radial  solutions of
$\Delta u+f(u)=0\in\mathbb{R}^n,$} Arch. Ration. Mech. Anal. 99
(1987), 115-145.

\bibitem{MTW}K.\ McLeod, W.\ C.\ Troy and F.\ B.\ Weissler, {\em Radial solutions of $\Delta u+f(u)=0$ with
prescribed numbers of zeros,} J. Differential Equations 83(2)
(1990), 368-378.

\bibitem{MO} W.\ Magnus and F.\ Oberhettinger,
{\em Formulas and Theorems for the Special Functions of Mathematical
Physics,} Springer-Verlag, Berlin and New York, 1966.

 \bibitem{mer-rap} F. Merle and P.  Raphael,  {\em On a sharp lower bound on the blow-up rate for the L2 critical nonlinear Schrödinger equation},  J. Amer. Math. Soc. 19 (2006), no. 1, 37–90 

\bibitem{NT}W.-M.\ Ni and I.\ Takagi, {\em On the shape of least-energy solutions to a semilinear Neumann problem,}
Comm. Pure Appl. Math., 44 (1991), no. 7, 819-851.

\bibitem{O}Y.\ G.\ Oh, {\em On positive multi-bump states of nonlinear Schr$\ddot{o}$dinger equation under multiple
well potentials}, Comm. Math. Phys., 131 (1990), 223-253.

\bibitem{PF1}M.\ del Pino, and P.\ Felmer, {\em Semi-classical states for nonlinear Schr\"{o}dinger equation}, J. Funct. Anal.,149 (1997), 245-265.

\bibitem{PF2}
M.\ del Pino and P.\  Felmer, {\em Semi-classical states of
nonlinear Schr\"odinger equations: a variational reduction method.}
Math. Ann. 324 (2002), no. 1, 1-32.

 \bibitem{NS} K. Nakanishi and W. Schlag, {\em  Global dynamics above the ground state for the nonlinear Klein-Gordon equation without a radial assumption}. Arch. Ration. Mech. Anal. 203 (2012), no. 3, 809–851.
 
\bibitem{P} S.\ Pohozaev, {\em Eigenfunction of the equation $\Delta
u+f(u)=0,$} Soviet. Math. Dokl., 6 (1965), 1408-1411.

\bibitem{PS} L.\ Peletier and J.\ Serrin, {\em Uniqueness of positive solutions of quasilinear equations,} Arch. Ration. Mech. Anal., 81 (1983), 181-197.

\bibitem{rap1} P. Raphaël, {\em On the singularity formation for the nonlinear Schrödinger equation.} Evolution equations, 269–323, Clay Math. Proc., 17, Amer. Math. Soc., Providence, RI, 2013.



\bibitem{ST}
J.\ Serrin and M.\ Tang, {\em Uniqueness of ground states for
quasilinear elliptic equations,} Indiana Univ. Math. J. 49 (2000),
897--923.

\bibitem{struwe} M. Struwe, {\em Multiple solutions of differential equations without the Palais-Smale condition}, Math. Ann. 261, (1982), 399-412.

\bibitem{wdirichlet} J.\ Wei, {\em On the construction of single-peaked solutions to a singularly perturbed elliptic Dirichlet problem}, J. Diff. Eqns 129 (1996), no.2, 315-333.

\bibitem{W} J.\ Wei, {\em Uniqueness and critical spectrum of boundary spike solutions},   Proc. Royal Soc. Edinburgh, Section A (Mathematics) 131 (2001), 1457-1480.

\end{thebibliography}
\end{document}